\documentclass[11pt,reqno,bold]{paper}
\usepackage{geometry}
\usepackage{amsmath}
\usepackage{amssymb}
\usepackage{amsfonts}
\usepackage{mathrsfs}
\usepackage{dsfont}
\usepackage{enumerate}
\usepackage{comment}
\usepackage{amsthm}
\usepackage{color}
\definecolor{darkblue}{rgb}{0.0,0.0,0.4}
\usepackage[pdfauthor={Al Balushi, Tsogtgerel},pdftitle={Adaptive finite element method for the biharmonic problem},pdfsubject={adaptive methods, convergence rate, biharmonic problem},colorlinks,citecolor=darkblue,linkcolor=darkblue,urlcolor=darkblue]{hyperref}

\newtheorem{theorem}{Theorem}[section]

\newtheorem{lemma}[theorem]{Lemma}
\newtheorem{corollary}[theorem]{Corollary}
\theoremstyle{definition}
\theoremstyle{remark}
\newtheorem{remark}[theorem]{Remark}

\usepackage{algorithmicx}
\usepackage[ruled]{algorithm}
\usepackage{algpseudocode}

\excludecomment{comment:quasi}

\newcommand{\ms}[1]{{\mathbf{#1}}}

\newcommand{\ds}[1]{{\mathds{#1}}}
\newcommand{\bs}[1]{{\boldsymbol{#1}}}
\renewcommand{\rm}[1]{{\mathrm{#1}}}
\newcommand{\bb}[1]{{\mathbb{#1}}}
\renewcommand{\cal}[1]{{\mathcal{#1}}}
\newcommand{\scr}[1]{{\mathscr{#1}}}
\newcommand{\n}{\ms{n}}
\newcommand{\diff}{\partial}
\newcommand{\grad}{\nabla}
\newcommand{\Lap}{\Delta}

\renewcommand{\dfrac}[2]{\frac{\diff #1}{\diff #2}}
\newcommand{\semi}[2]{{\left|#1\right|}_{#2}}
\newcommand{\norm}[2]{\|#1\|_{#2}}
\newcommand{\inner}[1]{\left\langle#1\right\rangle}
\newcommand{\enorm}[2]{|\!|\!|#1|\!|\!|_{#2}}
\newcommand{\jump}[2]{\bb{J}_{#2}\!\left(#1\right)}

\newcommand{\supp}{\mathrm{supp}\,}
\newcommand{\diam}{\rm{diam}\,}
\newcommand{\lapprox}{\preceq}

\newcommand{\eps}{\varepsilon}
\renewcommand{\u}{u}
\renewcommand{\v}{v}
\newcommand{\w}{w}
\newcommand{\e}{e}
\newcommand{\U}{U}
\newcommand{\V}{V}
\newcommand{\W}{W}
\newcommand{\X}{\bb{X}}
\renewcommand{\a}{a}

\newcommand{\marked}{\scr{M}}
\newcommand{\face}{\tau}
\newcommand{\cell}{\tau}
\newcommand{\edge}{\sigma}
\newcommand{\basis}{B}
\newcommand{\diffe}{\diff_\edge}
\newcommand{\eff}{f}

\newcommand{\est}{\eta}
\newcommand{\estP}{\eta_\mesh}
\newcommand{\up}{\U}
\newcommand{\upp}{\U_\ast}
\newcommand{\ep}{e_\mesh}
\newcommand{\epp}{e_{\mesh_\ast}}

\newcommand{\Ip}{I_\mesh}

\newcommand{\Ho}{H_0^2}
\renewcommand{\H}{H^2}

\newcommand{\mesh}{P}
\newcommand{\bedges}{\cal{G}}
\newcommand{\edges}{\cal{E}}
\newcommand{\Xp}{\bb{X}_\mesh}
\newcommand{\Xpp}{\bb{X}_{\mesh_\ast}}
\newcommand{\ap}{a_\mesh}
\newcommand{\app}{a_{\mesh_\ast}}
\newcommand{\Ep}{\scr{E}_\mesh}
\newcommand{\Epp}{\scr{E}_{\mesh_\ast}}
\newcommand{\Op}{\cal{L}}
\newcommand{\Pip}{\Pi_\mesh}
\newcommand{\Pipp}{\Pi_{\mesh_\ast}}
\newcommand{\ellf}{\ell_f}
\newcommand{\osc}{\rm{osc}}
\newcommand{\etap}{\eta_\mesh}
\newcommand{\etapp}{\eta_{\mesh_\ast}}
\newcommand{\Ccoer}{C_\rm{coer}}
\newcommand{\Ccont}{C_\rm{cont}}

\newcommand{\cnorm}{c_\rm{norm}}
\newcommand{\cshape}{c_\rm{shape}}
\newcommand{\Crel}{C_\rm{rel}}
\newcommand{\Cdrel}{C_\rm{dRel}}
\newcommand{\Ceff}{C_\rm{eff}}
\newcommand{\Cest}{C_\rm{est}}
\newcommand{\qest}{q_\rm{est}}
\newcommand{\Clip}{C_\rm{lip}}
\newcommand{\Ccomp}{C_\rm{comp}}

\newcommand{\ctrace}{d_0}
\newcommand{\cinv}{d_1}
\newcommand{\cdtrace}{d_2}
\renewcommand{\c}{d_3}
\newcommand{\cc}{d_4}
\newcommand{\cproj}{c_1}

\title{A convergent adaptive spline-based finite element method for the bi-Laplace operator using Nitsche's method}
\author{Ibrahim Al Balushi}
\institution{McGill University}
\date{\today}

\begin{document}
\maketitle
\begin{abstract}
We establish the convergence of an adaptive spline-based finite element method of a fourth order elliptic problem with weakly-imposed Dirichlet boundary conditions using polynomial B-splines.
\end{abstract}
\section{Introduction}
Design of optimal meshes for finite element analysis is a topic of extensive research going back to the early seventies.
Among the rich variety of strategies explored, the first mathematical framework for automatic optimal mesh generation was laid in the seminal work of Babushka and Rheinboldt \cite{babuvvska1978error}.
They introduce a class of computable a posteriori error estimates for a general class of variational problems which provides a strategy of extracting localized approximations of the numerical error of the exact solution.
The derived computable estimates are shown to form an upper bound with the numerical error, justifying the validity of the estimator, and a lower bounds which ensures efficient refinement; i.e, refinement where only necessar.
This established equivablence with the numerical error eluded to promising potential of practicality and robustness.
The theoretical results led to a heuristic characterization of optimal meshes through the even distribution a posteriori error quantities over all mesh elements, providing a blue-print for adaptive mesh generation.
The first detailed discription and performance analysis for a simple and accessible a posteriori error estimator was conceived for a one-dimensional elliptic and parabolic second-order Poisson-type problems in \cite{babuvska1978posteriori}.
The analysis was significantly improved in \cite{verfurth1994posteriori} for two-dimensional scenarios and develops numerous techniques used in the derivation of a posteriori estimates untill today.\\
\\
The first convergence was given in \cite{bubuvska1984feedback} in one-dimensions and later extended to two-dimenions by Dorfler \cite{dorfler1996convergent}.
Combing the advent of a novel \emph{marking} strategy and \emph{finess} assumptions of the initial mesh, \cite{dorfler1996convergent}.
Element marking directs the refinement procedure to select user-specified ratio of elements with highest error indicators relative to the total estimation.
The proposed strategy is later shown to be optimal \cite{cascon2008quasi}.
The initial mesh assumption was placed to ensure problem datum, such as source function and boundary values, are sufficiently resolved for detection by the solver.
The aim is to ensure error reduction of the estimator and thus monotone convergence of numerical error is achieved through contraction of consecutive errors in energy norm at every step.
but at the expense of potential over-refinement of the initial mesh. This was achieved using a local counter part of the efficiency estimate described above.\\
\\
Morin et al \cite{morin2000data} came to the realization that the averaging of the data has an unavoidable interference with the estimator error reduction irrespective of quadrature and was due to the avergaing of finer feathures of data brought by finite-dimensional approximations.
This averaging was quantified into an \emph{oscillation term}, a quantity tightly related to the criterion used in the initial mesh assumption used in \cite{dorfler1996convergent} but it provided a sharper representation of the underlying issue.
As a result the initial mesh assumption was removed in \cite{morin2000data} and replaced with a Dolfer-type marking criterion, \emph{separate marking}, for the data oscillation.
Unfortunately, the relaxation of a fine initial mesh had the unintended consequence of losing the strict monotone behaviour of numerical error decay.
It led to the introduction of the \emph{interior node property} to ensure error reduction with every step so as to ensure two consequence solutions will not be the same unless they are equal to the exact one; but at the expense of introducing over refinement. Each marked elementin in a two-dimensional triangular mesh undergoes three bisections ensuring an interior node, which furnishes us with a local lower bound and thus recovers strict error reduction with every iteration.
The results were extended to saddle-problems in \cite{morin2002convergence} and genarlized into abstract Hilbert setting in \cite{morin2008basic}.\\
\\
For the better part of the 2000's Morin-Nochetto-Sierbert algorithm (MNS) champoined adaptive finite element methods AFEM of linear elliptic problems after which the analysis was refined by Cascon \cite{cascon2008quasi} in concrete setting where they did not rely on a local lower bound for convergence which led to the ultimate removal of the costly interior node property and separate marking for oscillation. This was done while achieving quasi-optimal mesh complexity; see \cite{binev2004adaptive},\cite{stevenson2005optimal} for dertails. It was realized that strict reduction of error in energy norm cannot be gaurenteed whenever consequetive numerical solutions coinside but strict monotone decay is be obtain with respect to a suitable \emph{quasi-norm}. The result was extended to abstract Hilbert by Siebert \cite{siebert2010converg} which is now widely considered state of the art analysis of AFEM among the adaptive community and It hinges on the following ingredients: a global upper bound justifying the validity of the a posteriori  estimator, a Lipschitz property of the estimator as a function on the discrete finite element trail space indicating suitable sensitivity in variation within the trial space, a Pythgurous-type relation furnished by the variational and discrete forms and any suitable making strategy akin to that of Dofler; one that aims to equally distributes the elemental error estimates.\\
\\
Standard finite element methods (FEM) are based on triangular mesh partitions which have proven to be very robust at discretizing domains with complex geometry and are well-suited to problems requiring $H^1$ conforming shape functions. Higher degrees of smoothness across the element inferfaces is however much more involved. In recent years, with the emergence of \emph{isogeometric analysis} (IGA); see Hughes et al \cite{hughes2005isogeometric}, much attention has been directed at polynomial spline-based methods. Motivation began with the desire to integrate the CAD and analysis stages of design. As an immediate bonus, ploynomial spline-based meshes makes it easy to construct arbitrarily high orders of smoothness due to the mesh recutangular structure. In addition, NURB curves are robust at capturing curved geometries without the accumilation interpolation errors arising from standard trinagular-based FEM meshes.
However there is a drawback of using smooth spline-based bases for there is difficulty in prescribing essential boundary conditions (BC). Unlike nodal-based finite elements, smooth polynomial splines arrising from B-splines or NURBS are typically non-interpolatory which makes prescriptions of Dirichlet boundary conditions challenging and lead to highly oscillatory errors near the boundary \cite{bazilevs2007weak}.
In an earlier paper by Nitsche \cite{nitsche1971va} a weaker prescription of the boundary conditions is carried where BC are incorprated in the variational form rather than imposing it directly onto the discrete space \cite{stenberg1995some}. This idea hass been recently applied to the bi-Laplace operator \cite{embar2010imposing} using spline-based bases. An initial a posteriori analysis with this framework has been carried in \cite{juntunen2009nitsche} where the reliability and efficiency estimates are derived for the Poisson problem. However, the estimates included weighted boundary terms with negative powers and relied on a \emph{saturation assumption}. Recently, the idea has been employed in the treatment of a fourth-order elliptic problem appearing in geophysical flows
\cite{kim2015b},\cite{al2018adaptivity} with the added improvement that terms with negative powers were shown to be irrelevant much like in the case of adaptive discontinuous Galerkin methods (ADFEM)\cite{bonito2010quasi}.
While the analyses of \cite{morin2000data},\cite{morin2002convergence} justifies the use of the saturation assumption using a local lower bound in the Poisson problem, no such estimate is yet available for its fourth-order counterpart. In this work we aim to remove the saturation assumption as well as provide a convergence proof standard in residual-based AFEM literature of \cite{cascon2008quasi}. Many of the ideas are borrowed from the treament of ADFEM methods in \cite{bonito2010quasi} highlighting the similarity in nature of both mehods, theoreticaly as well as numerically.

Let $\Omega$ be a bounded domain in $\bb{R}^2$ with polygonal boundary $\Gamma$.
For a source function $f\in L^2(\Omega)$ we consider the following homogenous Dirichlet boundary-valued problem
\begin{eqnarray}\label{eq:pde}
\Op\u(x):=\Lap^2\u(x)=f(x)&&\text{in}\ \Omega\\
u=\diff u/\diff\nu=0&& \text{on}\ \Gamma.\nonumber
\end{eqnarray}
The adaptive procedure iterates over the following modules
\begin{equation}\label{eq:afem}
\boxed{\ms{SOLVE}}\longrightarrow\boxed{\ms{ESTIMATE}}\longrightarrow\boxed{\ms{MARK}}\longrightarrow\boxed{\ms{REFINE}}
\end{equation}
The module $\textbf{SOLVE}$ computes a hierarchical polynomial B-spline (HB) approximation $\U$ of the solution $\u$ with respect to a hierarchical partition $\mesh$ of $\Omega$.
A detailed discussion on the nature of such partitions will be carried in Section~\ref{sec:part}
For the module \textbf{ESTIMATE}, we use a residual-based error estimator $\eta_\mesh$ derived from the a posteriori analysis in Section~\ref{sec:post}.
The module \textbf{MARK} follows the D\"olfer marking criterion of \cite{dorfler1996convergent}.
Finally, the module \textbf{REFINE} produces a new refined partition $\mesh_\ast$ satisfying certain geometric constraints described in Section~\ref{sec:part} to ensure sharp approximation.\\
\\
\subsection{Notation}
We begin by laying out the notational conventions and function space definitions used in this presentation.
Let $\mesh$ be a partition of domain $\Omega$ consisting of square cells $\cell$ following the structure described in \cite{vuong2011hierarchical},\cite{al2018adaptivity}.
Denote the collection of all interior edges of cells $\cell\in\mesh$ by $\edges_\mesh$ and all those along the boundary $\Gamma$ are to be collected in $\bedges_\mesh$.
We assume that cells $\tau$ are open sets in $\Omega$ and that edges $\sigma$ do not contain the vertices of its affiliating cell.
Let $\diam(\omega)$ be the longest length within a Euclidian object $\omega$ and set $h_\face:=\diam(\face)$ and $h_\edge:=\diam(\edge)$. Then let the mesh-size $h_\mesh:=\max_{\face\in\mesh}h_\face$.
Define the boundary mesh-size function $h_\Gamma\in L^\infty(\Gamma)$ by
\begin{equation}
h_\Gamma(x)=\sum_{\sigma\in\bedges_\mesh}h_\sigma\ds{1}_\sigma(x),
\end{equation}
where the $\ds{1}_\sigma$ are the indicator functions on boundary edges.
Let $H^s(\Omega)$, $s>0$, be the fractional order Sobolev space equipped with the usual norm $\norm{\cdot}{H^s(\Omega)}$; see references \cite{adams1975sobolev},\cite{grisvard2011elliptic}.
Let $H^s_0(\Omega)$ be given as the closure of the test functions $C_c^\infty(\Omega)$ in $\norm{\cdot}{H^s(\Omega)}$.
The semi-norm $\semi{\cdot}{H^s(\Omega)}$ defines a full norm on $H^s_0(\Omega)$ by virtue of Poincar\'e's inequality. Moreover, the semi-norm $\norm{\Lap\cdot}{L^2(\Omega)}$ defines a norm on $H^2_0(\Omega)$.
Let
\begin{equation}
\bb{E}(\Omega)=\left\{\v\in H_0^2(\Omega):\cal{L}\v\in L^2(\Omega)\right\}.
\end{equation}
By $H^{-2}(\Omega)=(H^{2}(\Omega))'$ the dual of $H^{2}(\Omega)$ with the induced norm
\begin{equation}
\norm{F}{H^{-2}(\Omega)}=\sup_{\v\in H^2(\Omega)}\frac{\inner{F,\v}}{\norm{\v}{H^2(\Omega)}}.
\end{equation}
We will be making use of the following mesh-dependent (semi)norms on $H^2(\Omega)$ which we employ in Nitsche's discretization:
\begin{equation}
\norm{\v}{s,\mesh}^2=\sum_{\edge\in\bedges_\mesh}h_\sigma^{-2s}\norm{\v}{L^2(\edge)}^2,
\end{equation}
\begin{equation}\label{eq:meshnorm}
\enorm{\v}{\mesh}^2=\norm{\Lap\v}{L^2(\Omega)}^2+\gamma_1\norm{\v}{3/2,\mesh}^2+\gamma_2\left\norm{\textstyle\dfrac{\v}{\nu}\right}{1/2,\mesh}^2,
\end{equation}
with $\gamma_1$ and $\gamma_2$ are suitably large positive stabilization parameters.
Finally, we denote  $a\lapprox b$ to indicate $a\leq Cb$ for a constant $C>0$ assumed to be independent of any notable parameters unless otherwise stated.
\begin{remark}
It will be clear in Lemma~\ref{lem:normequiv} that for any mesh $\mesh$, the semi-norm $\enorm{\cdot}{\mesh}$ is in fact a full norm on $H^2(\Omega)$.
We ensure that in the limit of procedure \eqref{eq:afem} the numerical error measured in \eqref{eq:meshnorm} remains representative of the behaviour of error measurement in $H^2(\Omega)$ despite the presence of negative powers of $h$.
\end{remark}
\subsection{Problem setup}
The natural weak formulation to the PDE \eqref{eq:pde} reads
\begin{equation}\label{eq:cwp}
\text{Find}\ \u\in\Ho(\Omega)\ \text{such that}\ \a(\u,\v)=\ellf(\v)\ \text{for all}\ \v\in\Ho(\Omega),
\end{equation}
where $\a:\Ho(\Omega)\times\Ho(\Omega)\to\bb{R}$ is be the bilinear form $\a(\u,\v)=(\Lap\u,\Lap\v)_{L^2(\Omega)}$ and $\ellf(\v)=(f,\v)_{L^2(\Omega)}$.
The energy norm $\enorm{\cdot}{}:=\sqrt{\a(\cdot,\cdot)}\equiv\norm{\cdot\Lap}{L^2(\Omega)}$ is one for which the form $\a$ is continuous and coercive on $H^2_0(\Omega)$, with unit proportionality constants, and the existence of a unique solution is therefore ensured by Babuska-Lax-Milgram theorem.
The variational formulation \eqref{eq:cwp} is consistent with the PDE \eqref{eq:pde} under sufficient regularity considerations; if $\u\in \bb{E}(\Omega)$ satisfies \eqref{eq:cwp} then $\u$ satisfies \eqref{eq:pde} in the classical sense by virtue of the Du Bois-Reymond lemma.
The space of piecewise polynomials of degree $r\ge2$ defined on a partition $\mesh$ will be given by
\begin{equation}
\cal{P}^r_\mesh(\Omega)=\prod_{\tau\in\mesh}\bb{P}_r(\tau).
\end{equation}
Assuming we have at our disposal a polynomial B-spline space $\Xp\subset\cal{P}_\mesh^r(\Omega)\cap H^2_0(\Omega)$ then an immediate discrete problem reads
\begin{equation}\label{eq:cdp}
\text{Find}\ \U\in\Xp\ \text{such that}\ \a(\U,\V)=\ellf(\V)\ \text{for all}\ \V\in\Xp.
\end{equation}
The corresponding linear system is numerically stable and consistent with \eqref{eq:cwp} in the sense that $\a(\u,\V)=\ellf(\V)$ for every $V\in\Xp$ and therefore we are provided with Galerkin orthogonality:
\begin{equation}\label{eq:cgo}
\a(\u-\U,\V)=0\quad\forall\V\in\Xp.
\end{equation}
Moreover, the spline solution to \eqref{eq:cdp} will serve as an optimal approximation to $\u$ in $\Xp$ with respect to $\enorm{\cdot}{}$:
\begin{equation}
\label{eq:result:lem:ccl}
\enorm{\u-\U}{}\leq\inf_{\V\in\Xp}\enorm{\u-\V}{}.
\end{equation}
The discretization given in \eqref{eq:cdp} requires prescription of the essential boundary values into the discrete spline space $\Xp$, and as mentioned earlier, this poses difficulty when considering non-homogenous boundary conditions due to the non-iterpolatory nature of high-order smoothness B-splines.
Therefore from now on we will depart from a boundary-value conforming discretization and assume that the spline space $\Xp\subset\cal{P}_\mesh^r(\Omega)\cap H^2(\Omega)$ no longer satisfies the boundary conditions and instead impose them weakly.
In the previous work \cite{al2018adaptivity} the following mesh-dependent bilinear form $\ap:\Xp\times\Xp\to\bb{R}$ is used to formulate Nitsche's discretization:
\begin{equation}\label{eq:ndp}
\text{Find}\ \U\in\Xp\ \text{such that}\ \ap(\U,\V)=\ellf(\V)\ \text{for all}\ \V\in\Xp.
\end{equation}
where
\begin{equation}\label{eq:oldnitbilin}
\begin{split}
\ap(\U,\V)&=\a(\U,\V)-\int_\Gamma\left(\textstyle\Lap\U\dfrac{\V}{\nu}+\Lap\V\dfrac{\U}{\nu}\right)
+\gamma_1\int_\Gamma h_\Gamma^{-3}\U\V
\\
&+\int_\Gamma\left(\textstyle\dfrac{\Lap\U}{\nu}\V
+\dfrac{\Lap\V}{\nu}\U\right)
+\gamma_2\int_\Gamma h_\Gamma^{-1}\textstyle\dfrac{\U}{\nu}\dfrac{\V}{\nu}.
\end{split}
\end{equation}
The discrete problem of \eqref{eq:ndp} with bilinear form \eqref{eq:oldnitbilin} is consistent with its continuous counterpart \eqref{eq:cwp} and quasi-optimal a priori error estimates have been realized; see \cite{al2018adaptivity}.
Unfortuantely, much like the analysis carried in \cite{juntunen2009nitsche},\cite{al2018adaptivity}, all \emph{a posteriori} estimates relied on the artificial so-called saturation assumption.
Here we will consider a modified version of the bilinear form \eqref{eq:oldnitbilin} which extends the domain of $\ap$ to all of $H^2(\Omega)$. This will enable us to remove the saturation assumption while carrying complete convergence analysis, and in an upcoming publication, an optimality analysis. Moreover, for discrete arguments the new bilinear form reduces back to \eqref{eq:oldnitbilin} . This will however be at the expense of consistency where we will no longer have access to \eqref{eq:cgo}. It will be shown that this obstacle is manageable and all desired conclusions will be met at the price of more delicate treatment.\\
\\
Let $\Pip:L^2(\Omega)\to\cal{P}^{r-2}_\mesh(\Omega)$ be the $L^2$-orthogonal projection operator given by
\begin{equation}
\forall\v\in L^2(\Omega),\ \Pip\v\in\cal{P}^{r-2}_\mesh(\Omega)
\ \text{such that}\
\int_\Omega\Pi_\mesh\v q=\int_\Omega \v q\quad\forall q\in\cal{P}_\mesh^{r-2}(\Omega).
\end{equation}
Instead of \eqref{eq:oldnitbilin} we consider the bilinear form $\ap:H^2(\Omega)\times H^2(\Omega)\to\mathbb{R}$
\begin{equation}\label{eq:nbf}
\begin{split}
\ap(\u,\v)&=\a(\u,\v)
-\int_\Gamma\left(\textstyle\Pip(\Lap\u)\dfrac{\v}{\nu}+\Pip(\Lap\v)\dfrac{\u}{\nu}\right)
+\gamma_1\int_\Gamma h_\Gamma^{-3}\u\v
\\
&+\int_\Gamma\left(\textstyle\dfrac{\Pip(\Lap\u)}{\nu}\v
+\dfrac{\Pip(\Lap\v)}{\nu}\u\right)
+\gamma_2\int_\Gamma h_\Gamma^{-1}\textstyle\dfrac{\u}{\nu}\dfrac{\v}{\nu}.
\end{split}
\end{equation}
The problem we will consider will read as \eqref{eq:ndp} but now with $\ap$ defined by \eqref{eq:nbf}.
To simplify notation we define
\begin{equation}
\lambda_\mesh(u,v):=\int_\Gamma\left({\textstyle\dfrac{\Pip(\Lap\u)}{\nu}\v-\Pip(\Lap\u)\dfrac{\v}{\nu}}\right),
\quad
\lambda_\mesh^\ast(u,v):=\int_\Gamma\left(\textstyle\u\dfrac{\Pip(\Lap\v)}{\nu}-\dfrac{\u}{\nu}\Pip(\Lap\v)\right).
\end{equation}
%
%
\begin{remark}%
In the absence of ambiguity we drop the partition sign from $\Pip$, $\lambda_\mesh$ and $\lambda^\ast_\mesh$.
\end{remark}%
The solution $\u$ to \eqref{eq:cwp} does not satisfy the modified problem \eqref{eq:ndp}.
However, the inconsistency brought by the projection terms will be resolved  asymptotically.
To quantify the inconsistency for $\u\in\bb{E}(\Omega)$, let $\Ep(\u)\in H^{-2}(\Omega)$ be given by
\begin{equation}\label{eq:it}
\inner{\Ep(\u),\v}=\int_\Gamma\left(\textstyle\dfrac{\Pi(\Lap\u)}{\nu}-\dfrac{\Lap\u}{\nu}\right)\v-
\int_\Gamma\left(\textstyle\Pi(\Lap\u)-\Lap\u\right)\textstyle\dfrac{\v}{\nu},
\quad\v\in H^2(\Omega).
\end{equation}
\begin{lemma}[Inconsistency]\label{lem:ac}
If $\u\in\bb{E}(\Omega)$ is the solution to \eqref{eq:cwp} then
\begin{equation}\label{eq:res:lem:ac}
\ap(\u,\v)=\ellf(\v)+\inner{\Ep(\u),\v}\quad\forall\v\in H^2(\Omega).
\end{equation}
\end{lemma}
\begin{proof}
Integrate by parts to get
\begin{equation}
\begin{split}
\ap(\u,\v)-\ellf(\v)&=\int_\Omega\left(\Op\u-f\right)\v
+\int_\Gamma{\Lap\u\textstyle\dfrac{\v}{\nu}}-\int_\Gamma{\textstyle\dfrac{\Lap\u}{\nu}}\v\\
&+\int_\Gamma{\textstyle\dfrac{\Pi(\Lap\v)}{\nu}}\u-\int_\Gamma\Pi(\Lap\v){\textstyle\dfrac{\u}{\nu}}
-\int_\Gamma{\textstyle\Pi(\Lap\u)\dfrac{\v}{\nu}}+\int_\Gamma\textstyle\dfrac{\Pi(\Lap\u)}{\nu},
\end{split}
\end{equation}
and
\begin{equation}
\int_\Omega\left(\Op\u-f\right)\v=\int_\Gamma{\textstyle\dfrac{\Pi(\Lap\v)}{\nu}}\u=\int_\Gamma\Pi(\Lap\v){\textstyle\dfrac{\u}{\nu}}=0\quad\forall\v\in H^2_0(\Omega),
\end{equation}
since $\u$ satisies the boundary valued differential equation \eqref{eq:pde}.
\end{proof}
\begin{remark}
It will be assumed from now on that the argument $\u\in\bb{E}(\Omega)$ in \eqref{eq:res:lem:ac} will aways be the continuous solution to \eqref{eq:cwp} and therefore we will drop the $(\u)$ from $\Ep(\u)$.
\end{remark}
\begin{remark}
Noting that $\Ho(\Omega)$ is in the kernel of $\Ep$, we see from \eqref{eq:res:lem:ac} that $\ap$ reduces to $\a$ and the discrete formulation \eqref{eq:ndp} is in fact consistent with \eqref{eq:cwp} whenever test functions $\v$ satisfy the boundary conditions.
\end{remark}

\subsection{The adaptive method}
We define the residual quantity $\scr{R}_\mesh\in H^{-2}(\Omega)$ by
\begin{equation}\label{eq:residual}
\inner{\scr{R}_\mesh,\v}=\ap(\u-\U,\v),\quad\v\in H^2(\Omega).
\end{equation}
In view of \eqref{eq:res:lem:ctynbf} and \eqref{eq:res:lem:csvnbf} we readily have sharp a posteriori estimates for $\u-\U$
\begin{equation}
\Ccont^{-1}\norm{\scr{R}_\mesh}{H^{-2}(\Omega)}\leq\enorm{\u-\U}{\mesh}\leq\Ccoer^{-1}\norm{\scr{R}_\mesh}{H^{-2}(\Omega)}.
\end{equation}
The quantity $\norm{\scr{R}_\mesh}{H^{-2}(\Omega)}$ is computable since it only depends on available discrete approximation of solution $\u$. However, exact computation of $H^{-2}$-norms is of infinite dimensional complexity making it computationally infeasible.
Instead we follow the techniques devised in \cite{verfurth1994posteriori},\cite{ainsworth2011posteriori}.
We now discuss the modules \textbf{SOLVE}, \textbf{ESTIMATE},
\textbf{MARK} and \textbf{REFINE} in detail.
\subsubsection*{The module SOLVE}
\begin{equation}\label{eq:dnp}
\U=\ms{SOLVE}[\mesh,f]:\quad\text{Find}\ \U\in\Xp\ \text{such that}\ \ap(\U,\V)=\ellf(\V)\ \text{for all}\ \V\in\Xp.
\end{equation}
\subsubsection*{The module ESTIMATE}
\begin{equation}\label{eq:indicator}
\eta_\mesh^2(\V,\face)=h_\face^4\norm{f-\Op\V}{L^2(\tau)}^2
+\sum_{\sigma\subset\diff\tau}
\left(\textstyle h_\sigma^{3}\left\norm{\jump{\dfrac{\Lap\V}{\n_\sigma}}{\sigma}\right}{L^2(\sigma)}^2
+h_\sigma\norm{\jump{\Lap\V}{\sigma}}{L^2(\sigma)}^2\right)
\end{equation}
\begin{equation}\label{eq:estimator}
\eta_\mesh^2(\V,\omega)=\sum_{\tau\in \mesh:\tau\subset\omega}\eta_\mesh^2(\V,\face),\quad\omega\subseteq\Omega
\end{equation}
\begin{equation}
\{\eta_\tau:\tau\in\mesh\}=\ms{ESTIMATE}[\U,\mesh]:\quad\eta_\tau:=\eta_\mesh(\U,\U)
\end{equation}
The interior integrals measures the gap between $H^2$ and $H^4$ and the edge captures $H^{-2}$ componets of the approximation.
\subsubsection*{The module MARK}
We follow the Dorlfer marking strategy \cite{dorfler1996convergent}: For $0<\theta\leq1$,
\begin{equation}\label{markingstrategy}
\text{Find minimal spline set}\ \marked:\quad\sum_{\cell\in\marked}\eta^2_\mesh(\U,\cell)\ge\theta\sum_{\tau\in\mesh}\eta_\mesh^2(\U,\cell).
\end{equation}
\subsubsection*{The module REFINE}

\section{Finite element spline space}
Construction of our desired basis will be carried out in the following manner.
We wish to construct a conforming finite-element spline space
\begin{equation}
\Xp(\Omega)\subseteq\cal{P}^r_\mesh(\Omega)\cap C^1(\Omega),\quad(r\ge2).
\end{equation}
\subsection{Domain partition}\label{sec:part}
We consider the domain partition described in \cite{al2018adaptivity} which provides control on the support overlap of basis functions for the purpose of a stable, accurate and localized approximation to the solution.
This will enable us to achieve global quantification of the a priori and a posteriori error analyses without too much over estimation.
The setting ensures that perturbations in the coefficients of the approximating function should not influence the approximation globally.
This is achieved by constructing a spline basis that is locally linearly independent; any basis functions having support in $\Omega'\subset\Omega$ must be linearly independent with each other on $\Omega'$.
We define the support extension for a cell $\tau \in \mesh$ by
\begin{equation}
\omega_\tau=\{\tau'\in\mesh:\supp\beta\cap \tau'\neq\emptyset\implies\supp\beta\cap \tau\neq\emptyset\},
\end{equation}
indicating the collection of all supports for basis function $\beta$'s whose supports intersect $\tau$.
Analogously, we denote the support extension for an edge $\sigma\in\edges\cup\bedges$ by
\begin{equation}
\omega_\sigma=\{\tau\in\mesh:\supp\beta\cap\tau\neq\emptyset\implies\supp\beta\cap \tau\neq\emptyset,\ \sigma\subset\diff\tau\}.
\end{equation}
For a constant $\cshape>0$, depending only on the polynomial degree of the spline space, all considered partitions therefore will satisfy the shape-regularity constraints
\begin{eqnarray}\label{eq:sr}
\nonumber\max_{\tau\in\mesh}\#\left\{\tau\in\mesh:\tau\in\omega_\tau\right\}\leq\cshape&&\text{(finite-intersection property)},\\
\max_{\tau\in\mesh}\frac{h_\tau}{h_\sigma}\leq\cshape\ (\sigma\subset\diff\tau)
&&(\text{mesh grading}),\\
\nonumber\max_{\tau\in\mesh}\frac{\diam(\omega_\tau)}{h_\tau}\leq \cshape
&&(\text{quas-uniformity}).
\end{eqnarray}
We define the mesh-size $h_\mesh$ of $\mesh$ to be $h_\mesh=\max_{\tau\in\mesh}h_\tau$.\\
\subsection{Discrete spline space}
%
%

Recall the general trace theorem \cite{adams1975sobolev},\cite{grisvard2011elliptic} for cells $\tau\in\mesh$ and edges $\sigma\in\bedges_\mesh$ with $\sigma\subset\diff\tau$.
For a constant $\ctrace>0$
\begin{equation}\label{eq:GeneralTrace}
\norm{\v}{L^2(\sigma)}^2\leq \ctrace\left(h_\sigma^{-1}\norm{\v}{L^2(\tau)}^2+h_\sigma\norm{\grad\v}{L^2(\tau)}^2\right)\quad\forall\v\in H^1(\Omega).
\end{equation}
\begin{lemma}[Auxiliary discrete estimate]\label{lem:ie}
Let $\face\in\mesh$.
Then for $\cinv>0$, depending only on polynomial degree $r$, for $0\leq s\leq t\leq r+1$ we have
\begin{equation}\label{eq:inve:lem:ie}
\semi{\V}{H^t(\tau)}\leq\cinv h_\tau^{s-t}\semi{\V}{H^s(\tau)}\quad\forall\V\in\bb{P}_r(\tau),
\end{equation}
and if $\edge\subset\diff\face$, for a constant $\cdtrace>0$ we have
\begin{equation}\label{eq:dtrace:lem:ie}
\norm{\V}{L^2(\sigma)}\leq\cdtrace h_\sigma^{-1/2}\norm{\V}{L^2(\tau)}\quad\forall\V\in\bb{P}_r(\tau),
\end{equation}
where $\cdtrace:=\ctrace\max\{1,\cinv\}$.
\end{lemma}
\begin{remark}
The constants $\cinv,\ \ctrace,\ \cdtrace$ all depend on the polynomial degree and the reference cell or edge; $\hat{\tau}=[0,1]^2$ or $\hat{\sigma}=[0,1]$.
From now, for a simpler presentation of the analysis, we combined all these constants, and their powers into a unifying constant $c_\ast$
\end{remark}
%

\subsection{Approximation in $\Xp$}
Given a hierarchical B-spline basis $\{\basis_\lambda\}_{\lambda\in\Lambda}$, we obtain an $L^2$-dual basis $\{\psi_\lambda\}_{\lambda\in\Lambda}\Xp$ via the linear system
\begin{equation}
(\psi_\lambda,\basis_\mu)_{L^2(\Omega)}=\delta_{\lambda\mu}\quad\mu\in\Lambda.
\end{equation}
Then $\psi_\lambda(\basis_\mu)=\delta_{\lambda\mu}$ for every multi-index $\lambda$ and $\mu$.
The splines $\psi_\lambda$ will be viewed as linear functionals on $L^2(\Omega)$ via the identification $\psi_\lambda(f)=(\psi_\lambda,f)_{L^2(\Omega)}$, $f\in L^2(\Omega)$.
Let
\begin{equation}
J_\lambda=\{\mu\in\Lambda:\supp\basis_\mu\cap\supp\basis_\lambda\neq\emptyset\}.
\end{equation}
We estimate $\psi_\lambda$.
Let $T_\lambda$ be the affine transformation for which $\diam(T_\lambda(\supp \basis_\lambda)=1$.
Let $f\in L^2(\Omega)$ and define its pull-back $\hat{f}=f\circ T_\lambda$.
Also, express $\psi_\lambda$ as the linear combination of $\cal{B}_\mesh$ for some real coefficients $\{c_\mu\}_{\mu\in\Lambda}$.
\begin{equation}
|\hat{\psi}_\lambda(\hat{f})|
=\bigg|\sum_{\mu\in\Lambda}c_\mu\int_{\supp\hat{\basis}_\lambda}\hat{\basis}_\mu\hat{f}\bigg|
\leq\max_{\mu\in J_\lambda}|c_\mu|\bigg\|\sum_{\mu\in J_\lambda}\hat{\basis}_\mu\bigg\|_{L^2(T(\supp\basis_\lambda))}\norm{\hat{f}}{L^2(T_\lambda(\supp\basis_\lambda))}
\end{equation}
In view of \eqref{eq:sr} we have
\begin{equation}
\left\norm{\sum_{\mu\in J_\lambda}\hat{\basis}_\mu\right}{L^2(T(\supp\basis_\lambda))}
\lapprox \cshape|T(\supp\basis_\lambda)|^{1/2}\approx1.
\end{equation}
Moreover, the moduli of coefficients $\max_{\lambda\in\Lambda}|c_\lambda|$ are universally bounded for every $\psi_\lambda$.
Scaling back to $\supp\basis_\lambda$ yields
\begin{equation}
|\psi_\lambda(f)|
\lapprox|\supp\basis_\lambda|^{-1/2}\norm{f}{L^2(\Omega)}
\end{equation}
Define quasi-interpolation operator $I_\mesh:L^2(\Omega)\to\Xp(\Omega)$,
\begin{equation}
I_\mesh f
:=\sum_{\lambda\in\Lambda(\bs{\Omega})}\psi_\lambda(f)\phi_\lambda,
\quad f\in L^2(\Omega).
\end{equation}
The defined operator is a projection. Indeed, if $s=\sum_{\mu\in\Lambda}c_\mu\basis_\mu$ then
\begin{equation}
I_\mesh(s)
=\sum_{\lambda,\mu\in\Lambda}c_{\mu}\psi_\lambda(\basis_\mu)\basis_\lambda
=\sum_{\mu\in\Lambda}c_\mu\basis_\mu=s.
\end{equation}
\begin{lemma}
Let $\mesh$ be an admissible partition of $\Omega$. Then for every $\face\in\mesh$,
\begin{equation}
\norm{I_\mesh\v}{L^2(\cell)}\lapprox\cshape\norm{\v}{L^2(\omega_\cell)}\quad\forall\v\in L^2(\Omega).
\end{equation}
Moreover, for $0\leq t\leq s\leq r+1$.
\begin{equation}
\forall\cell\in\mesh,\quad\semi{\v-I_\mesh\v}{H^t(\cell)}\lapprox\cshape h_\cell^{s-t}\semi{\v}{H^s(\omega_\cell)}\quad\forall\v\in H^s(\Omega).
\end{equation}
\end{lemma}
\begin{proof}
By partition of unity we have
\begin{equation}
\begin{split}
\norm{I_\mesh\v}{L^2(\cell)}
&\leq|\cell|^{1/2}\max_{\lambda:\supp\basis_\lambda\cap\cell}|\psi_\lambda(\v)|\\
&\lapprox|\cell|^{1/2}\max_{\lambda:\supp\basis_\lambda\cap\cell}|\supp\basis_\lambda|^{-1/2}\norm{f}{L^2(\supp\basis_\lambda\cap\cell)}\\
&=|\cell|^{1/2}|\omega_\cell|^{-1/2}\norm{\v}{L^2(\omega_\cell)}
\end{split}
\end{equation}
however with shape-regularity we have $|\omega_\cell|^{-1/2}|\cell|^{1/2}\lapprox \cshape$.
In other words the operator $I_\mesh:L^2(\Omega)\to \Xp(\Omega)$ is a stable projection:
\begin{equation}
\norm{I_\mesh\v}{L^2(\cell)}\lapprox C_\rm{shape}\norm{\v}{L^2(\omega_\cell)}\quad\forall \v\in L^2(\Omega).
\end{equation}
Let $\v\in H^s(\Omega)$. Bramble-Hilbert lemma asserts that for $0\leq t\leq s\leq r+1$ we have
\begin{equation}
\inf_{\pi\in\bb{P}_{r}(G)}\semi{\v-\pi}{H^t(\cell)}\lapprox h_\cell^{s-t}\semi{\v}{H^s(G)},
\end{equation}
where the implicit constant depends on polynomial degree $r$ and $G\subseteq\Omega$.
Let $Q\in\bb{P}_r(\omega_\cell)$. We have
\begin{equation}
\semi{I_\mesh(\v-\pi)}{H^t(\cell)}\leq h_\cell^{-t}\norm{I_\mesh(\v-\pi)}{L^2(\cell)},
\end{equation}
by equivalence of norms in finite dimensions.
\begin{equation}
\begin{split}
\semi{\v-I_\mesh\v}{H^t(\cell)}
&\leq\inf_{\pi\in\bb{P}_r(\omega_\cell)}\left(\semi{\v-\pi}{H^t(\cell)}
+h_\cell^{-t}\norm{\pi-\v}{L^2(\omega_\cell)}\right),\\
&\lapprox\inf_{\pi\in\bb{P}_r(\omega_\cell)}\left(\semi{\v-\pi}{H^t(\cell)}
+h_\cell^{-t}\norm{\v-\pi}{L^2(\omega_\cell)}\right),\\
&\lapprox h_\cell^{s-t}\semi{\v}{H^s(\cell)}
+h_\cell^{-t}\rm{diam}(\omega_\cell)^s\semi{\v}{H^s(\omega_\cell)}.
\end{split}
\end{equation}
With shape-regularity we conclude
\begin{equation}
\forall\cell\in\cal{Q},\quad
\semi{\v-I_\mesh\v}{H^t(\cell)}
\lapprox C_\rm{shape}h_\cell^{s-t}\semi{\v}{H^s(\omega_\cell)}
\quad\forall\v\in L^2(\Omega)\cap H^s(\omega_\cell).
\end{equation}
\end{proof}
\section{A priori analysis for Nitsche's formulation}
In the coercivity analysis we follow the ideas presented in \cite{kim2015b},\cite{embar2010imposing} which are similar to the spirit of dG methods \cite{bonito2010quasi}.
\begin{lemma}[Continuity of $\ap$]\label{lem:ctynbf}
Let $\gamma_1,\gamma_2>0$ be given.
We have
\begin{equation}\label{eq:res:lem:ctynbf}
|\ap(\u,\v)|\leq\Ccont\enorm{\u}{\mesh}\enorm{\v}{\mesh}\quad\u,\v\in H^2(\Omega),
\end{equation}
with a constant $\Ccont>0$ independent of $\mesh$.
\end{lemma}
\begin{proof}
We begin with the interior integrals;
\begin{equation}
\a(\u,\v)\leq \norm{\Lap\u}{L^2(\Omega)}\norm{\Lap\v}{L^2(\Omega)}.
\end{equation}
As for the boundary terms,
\begin{equation}
\lambda^\ast(\u,\v)\leq\norm{\u}{L^2(\Gamma)}\left\norm{\textstyle\dfrac{\Pi(\Lap\v)}{\nu}\right}{L^2(\Gamma)}
+\left\norm{\textstyle\dfrac{\u}{\nu}\right}{L^2(\Gamma)}\left\norm{\Pi(\Lap\v)\right}{L^2(\Gamma)}
\end{equation}
\begin{equation}
\begin{split}
\lambda^\ast(\u,\v)&\lapprox\left\norm{h_\Gamma^{-3/2}\u\right}{L^2(\Gamma)}\left\norm{\Lap\u\right}{L^2(\Omega)}
+\left\norm{\textstyle h_\Gamma^{-1/2}\dfrac{\u}{\nu}\right}{L^2(\Gamma)}\left\norm{\Lap\v\right}{L^2(\Omega)}\\
&\lapprox \left(\norm{\u}{3/2,\mesh}+\left\norm{\textstyle\dfrac{\u}{\nu}\right}{1/2,h}\right)\left\norm{\Lap\v\right}{L^2(\Omega)}
\end{split}
\end{equation}
Similarily,
\begin{equation}
\lambda(\u,\v)\leq \left\norm{\Lap\u\right}{L^2(\Omega)}\left(\norm{\v}{3/2,\mesh}+\left\norm{\textstyle\dfrac{\v}{\nu}\right}{1/2,\mesh}\right).
\end{equation}
The stabilization terms are similarly controlled
\begin{equation}
\Sigma_\mesh(\u,\v)\leq\gamma_1\left\norm{\u\right}{3/2,\mesh}\left\norm{\v\right}{3/2,\mesh}+\gamma_2\left\norm{\textstyle\dfrac{\u}{\nu}\right}{1/2,\mesh}\left\norm{\textstyle\dfrac{\v}{\nu}\right}{1/2,\mesh}
\end{equation}
\end{proof}
%
\begin{lemma}\label{lem:L2estimates}
Let $\mesh$ be an admissible partition, let $\tau\in\mesh$ and let $\sigma\in\bedges_\mesh$ with $\sigma\subset\diff\tau$.
The projection operator $\Pi$ satisfies the following stability estimates:
\begin{equation}\label{eq:stability:lem:L2estimates}
\norm{\Pi\v}{L^2(\Omega)}\leq\norm{\v}{L^2(\Omega)},
\end{equation}
and
\begin{equation}\label{eq:inverse:lem:L2estimates}
\left\norm{\Pi\v\right}{L^2(\sigma)}\leq\c h_\sigma^{-1/2}\norm{\v}{L^2(\tau)}
\quad\text{and}\quad
\left\norm{\textstyle\dfrac{(\Pi\v)}{\n_\sigma}\right}{L^2(\sigma)}\leq\c h_\sigma^{-3/2}\norm{\v}{L^2(\tau)},
\end{equation}
holding for every $\v\in L^2(\Omega)$.
\end{lemma}
\begin{proof}
The stability estimate \eqref{eq:stability:lem:L2estimates} follows from orthogonality of the residual $\v-\Pi^r\v$ to $\Pi^r\v$.
To establish \eqref{eq:inverse:lem:L2estimates}, we will only prove the second one, as the first estimate follows similarly.
Let $\v\in L^2(\tau)$.
In view of Lemma \ref{lem:ie} and stability \eqref{eq:stability:lem:L2estimates}
\begin{equation}
\begin{split}
\left\norm{\textstyle\dfrac{(\Pi\v)}{\n_\sigma}\right}{L^2(\sigma)}^2&
\leq\cinv h_\sigma^{-2}\norm{\Pi\v}{L^2(\sigma)}^2\\
&\leq\cdtrace\cinv h_\sigma^{-3}\norm{\Pi\v}{L^2(\tau)}^2
\leq\cdtrace\cinv h_\sigma^{-3}\norm{\v}{L^2(\tau)}^2.
\end{split}
\end{equation}
\end{proof}
\begin{lemma}[Coercivity of $\ap$]\label{lem:csvnbf}
For suitably large stabilization parameters $\gamma_1$ and $\gamma_2$, there exists a constant $\Ccoer>0$ such that
\begin{equation}\label{eq:res:lem:csvnbf}
\Ccoer\enorm{\v}{\mesh}^2\leq\ap(\v,\v)\quad\forall\v\in H^2(\Omega).
\end{equation}
\end{lemma}
\begin{proof}
For $\delta_1,\,\delta_2>0$ we use Young's inequality to write
\begin{equation}
\begin{split}
\lambda(\v,\v)+\lambda^\ast(\v,\v)&\ge-\frac{1}{\delta_1}\norm{\v}{L^2(\Gamma)}^2
-\delta_1\left\norm{\textstyle\dfrac{\Pi(\Lap\v)}{\nu}\right}{L^2(\Gamma)}^2
-\frac{1}{\delta_2}\left\norm{\textstyle\dfrac{\v}{\nu}\right}{L^2(\Gamma)}^2
-\delta_2\left\norm{\Pi(\Lap\v)\right}{L^2(\Gamma)}^2.
\end{split}
\end{equation}
Together with the interior terms we have
\begin{equation}
\begin{split}
\ap(\v,\v)&\ge\norm{\Lap\v}{L^2(\Omega)}^2-\delta_1\left\norm{\textstyle\dfrac{\Pi(\Lap\v)}{\nu}\right}{L^2(\Gamma)}^2-\delta_2\left\norm{\Pi(\Lap\v)\right}{L^2(\Gamma)}^2\\
&+\left(1-\frac{1}{\gamma_1}-\frac{1}{\delta_1\gamma_1}\right)\gamma_1\left\norm{\psi\right}{3/2,\mesh}^2
+\left(1-\frac{1}{\delta_2\gamma_2}\right)\gamma_2\left\norm{\textstyle\dfrac{\v}{\nu}\right}{1/2,\mesh}^2
\end{split}
\end{equation}
\begin{equation}
\begin{split}
\ap(\v,\v)&\ge\left(1-\delta_1C\max_{\sigma\in\bedges}h_\sigma^{-3/2}-\delta_2C\max_{\sigma\in\bedges}h_\sigma^{-1/2}\right)\norm{\Lap\v}{L^2(\Omega)}^2\\
&+\left(1-\frac{1}{\gamma_1}-\frac{1}{\delta_1\gamma_1}\right)\gamma_1\left\norm{\v\right}{3/2,\mesh}^2
+\left(1-\frac{1}{\delta_2\gamma_2}\right)\gamma_2\left\norm{\textstyle\dfrac{\v}{\nu}\right}{1/2,\mesh}^2,
\end{split}
\end{equation}
\end{proof}
\begin{remark}
\end{remark}
\begin{lemma}[A priori error estimate for Nitsche's forumation]
\begin{equation}
\enorm{\u-\U}{\mesh}\leq\left(1+\frac{\Ccont}{\Ccoer}\right)\inf_{\V\in\Xp}\enorm{\u-\V}{\mesh}+\frac{1}{\Ccoer}\norm{\Ep}{H^{-2}(\Omega)}.
\end{equation}
\end{lemma}
\begin{proof}
From
\begin{equation}
\enorm{\u-\U}{\mesh}\leq\enorm{\u-\V}{\mesh}+\enorm{\V-\U}{\mesh}
\end{equation}
we will estimate $\enorm{\V-\U}{\mesh}$. Let $\W=\V-\U$,
\begin{equation}
\begin{split}
\Ccoer\enorm{\V-\U}{\mesh}^2&\leq\ap(\V-\u,\W)+\ap(\u-\U,\W)\\
&\leq\Ccont\enorm{\u-\V}{\mesh}\enorm{\W}{\mesh}+|\!\inner{\Ep,\W}\!|\\
&\leq\Ccont\enorm{\u-\V}{\mesh}\enorm{\W}{\mesh}+\norm{\Ep}{H^{-2}(\Omega)}\enorm{\W}{\mesh}
\end{split}
\end{equation}
which makes
\begin{equation}
\enorm{\u-\U}{\mesh}\leq\left(1+\textstyle\frac{\Ccont}{\Ccoer}\right)\enorm{\u-\V}{\mesh}+\textstyle\frac{1}{\Ccoer}\norm{\Ep}{H^{-2}(\Omega)}
\end{equation}
\end{proof}
We will assess the inconsistency and show that the formulation \eqref{eq:ndp} is in fact consistent asymptotically.
For this we will need some approximation tools.
\begin{lemma}\label{lem:L2error}
Let $\mesh$ be an admissible partition, let $\tau\in\mesh$ and let $\sigma\in\bedges_\mesh$ with $\sigma\subset\diff\tau$.
For a constant $\cproj>0$, depending only on $\cshape$,
if $0\leq t\leq s\leq r-1$ then
\begin{equation}\label{eq:result1:lem:L2error}
\semi{\v-\Pi(\v)}{H^t(\tau)}\leq \cproj h_\tau^{s-t}\semi{\v}{H^s(\tau)},
\end{equation}
and
\begin{equation}\label{eq:result2:lem:L2error}
\norm{\v-\Pi(\v)}{L^2(\sigma)}\leq\cproj h_\sigma^{s-1/2}\semi{\v}{H^s(\tau)},
\end{equation}
holding for every $\v\in\H(\Omega)$.
\end{lemma}
\begin{proof}
Let $1\leq t\leq s\leq r+1$ and let $\v\in\H(\Omega)$.
Let $\rho\in\bb{P}_r(\tau)$.
\begin{equation}
\begin{split}
\semi{\v-\Pi\v}{H^t(\tau)}&\leq\semi{\v-\rho}{H^t(\tau)}+\semi{\Pi(\rho-\v)}{H^t(\tau)}\\
&\leq\semi{\v-\rho}{H^t(\tau)}+\cinv h_\tau^{-t}\norm{\Pi(\rho-\v)}{L^2(\tau)}
\end{split}
\end{equation}
with the classical Bramble-Hilbert lemma we arrive at \eqref{eq:result1:lem:L2error} with $\cproj= (1+\cinv)c_\rm{HB}$ with $c_\rm{HB}>0$ is the proportionality constant of Bramble-Hilbert lemma.
Now in view of \eqref{eq:GeneralTrace}
\begin{equation}
\begin{split}
\norm{\v-\Pi\v}{L^2(\sigma)}^2\leq&\ctrace\left(h_\sigma^{-1}\norm{\v-\Pi\v}{L^2(\tau)}^2+h_\sigma\semi{\v-\Pi\v}{H^1(\tau)}^2\right)\\
&\ctrace\cproj\left(h_\sigma^{-1}h_\tau^{2s}\semi{\v}{H^s(\tau)}^2+
h_\sigma h_\tau^{2s-2}\semi{\v}{H^s(\tau)}^2\right)\\
&\leq \ctrace\cproj h_\sigma^{2s-1}\semi{\v}{H^s(\tau)}^2.
\end{split}
\end{equation}
\end{proof}
\begin{lemma}[Asymptotic consistency]\label{lem:AsymptoticConsistency}
If $\u\in\bb{E}(\Omega)$ is the solution to \eqref{eq:cwp} for which $\Lap\u\in H^s(\Omega)$, $s>0$, then for $\v\in H^2(\Omega)$,
\begin{equation}\label{eq:Resut:lem:AsymptoticConsistency}
\inner{\Ep,\v}\leq\cproj h_\mesh^s\left\norm{\Lap\u\right}{H^s(\Omega)}
\left(\norm{\v}{3/2,\mesh}+\left\norm{\textstyle\dfrac{\v}{\nu}\right}{1/2,\mesh}\right).
\end{equation}
\end{lemma}
\begin{proof}
\begin{equation}
\begin{split}
\inner{\Ep,\v}&=\int_\Gamma\left(\textstyle\dfrac{\Pi(\Lap\u)}{\nu}-\dfrac{\Lap\u}{\nu}\right)\v
-\int_\Gamma\left(\Pi(\Lap\u)-\Lap\u\right)\textstyle\dfrac{\v}{\nu}\\
&\leq\sum_{\sigma\in\bedges}\left\norm{\textstyle\dfrac{}{\n_\sigma}\left(\Pi(\Lap\u)-\Lap\u\right)\right}{L^2(\sigma)}\norm{\v}{L^2(\sigma)}
+\sum_{\sigma\in\bedges}\norm{\Pi(\Lap\u)-\Lap\u}{L^2(\sigma)}\left\norm{\textstyle\dfrac{\v}{\n_\sigma}\right}{L^2(\sigma)}
\end{split}
\end{equation}
In view of the projection error analysis of Lemma \eqref{lem:L2error}
\begin{equation}
\norm{\Pi(\Lap\u)-\Lap\u}{L^2(\sigma)}\leq\cproj h_\sigma^{s-1/2}\norm{\Lap\u}{H^s(\Omega)}
\end{equation}
and
\begin{equation}
\left\norm{\textstyle\dfrac{}{\n_\sigma}\left(\Pi(\Lap\u)-\Lap\u\right)\right}{L^2(\sigma)}\leq \cproj h_\sigma^{s-3/2}\norm{\Lap\u}{H^s(\Omega)}
\end{equation}
which leads us to the desired estimate.
\end{proof}
\begin{remark}
It is expected that if $\u\in\bb{E}(\Omega)$ then the regularity parameter $s$ would be at least $2$.
\end{remark}
Let $\Xp^0=\Xp\cap H_0^2(\Omega)$.
We characterize an orthogonal complement $\Xp^\perp$ to $\Xp^0$ using a projection operator $\pi^0_\mesh:\Xp\to\Xp^0$ defined by the linear problem
\begin{equation}
\pi_\mesh\V\in\Xp^0:\quad\ap(\W_0,\V-\pi^0_\mesh\V)=0\quad\forall\W_0\in\Xp^0.
\end{equation}
By setting $\pi_\mesh^\perp\V=\V-\pi_\mesh^0\V$ for any $\V\in\Xp$, we obtain a decompose for every finite-element spline
\begin{equation}
\V=\pi_\mesh^0\V+\pi_\mesh^\perp\V=:\V^0+\V^\perp\in\Xp^0\oplus\Xp^\perp\equiv\Xp
\end{equation}
with $\ap(\V^0,\W^\perp)=0$ for every pair $\V$ and $\W$.
We obtain the following result
\begin{lemma}\label{lem:InconsistentDecay}
If $\U$ is a finite-element spline solution to \eqref{eq:ndp},
\begin{equation}
\enorm{\U^\perp}{\mesh}\leq\frac{1}{\Ccoer}\left(\Ccont\inf_{\V\in\Xp^0}\enorm{\u-\V}{}+\enorm{\Ep}{H^{-2}(\Omega)}\right).
\end{equation}
\end{lemma}
\begin{proof}
Let $\U$ be the solution to \eqref{eq:dnp}.
\begin{equation}
\ap(\u-\U^\perp,\U^\perp)=\ap(\u-\U,\U^\perp)=\inner{\Ep,\U^\perp}.
\end{equation}
Then
\begin{equation}
\begin{split}
\Ccoer\enorm{\U^\perp}{\mesh}^2&\leq\ap(\u,\U^\perp)-\inner{\Ep,\U^\perp}
=\ap(\u-\V^0,\U^\perp)-\inner{\Ep,\U^\perp},
\end{split}
\end{equation}
and reconize that $\enorm{\v}{\mesh}\equiv\enorm{\v}{}$ whenever $\v\in H^2_0(\Omega)$.
\end{proof}
\begin{remark}
If $\u\in H^{s+2}$ for any $s>0$,
\begin{equation}
\enorm{\Ep}{H^{-2}(\Omega)}\leq \cproj h_\mesh^{s}\norm{\Lap\u}{H^s(\Omega)}
\quad\text{and}\quad
\inf_{\V\in\Xp^0}\enorm{\u-\V}{}\lapprox h_\mesh^{s}\enorm{\u}{}
\end{equation}
which makes for a constant $C$
\begin{equation}
\enorm{\U^\perp}{\mesh}\leq Ch_\mesh^s\norm{\Lap\u}{H^s(\Omega)}
\end{equation}
\end{remark}
\begin{lemma}\label{lem:normequiv}
\end{lemma}
\begin{proof}
$\v\in H^2$ then $\enorm{\v}{\mesh}<\infty$.
On the other hand if $\Lap\v\in L^2(\Omega)$ and the traces $\v|_\Gamma$ and $\dfrac{\v}{\nu}|_\Gamma$ belong to $L^2(\Gamma)$ we then $\v$ is a unique solution to a Poisson problem and
\begin{equation}
\end{equation}
so $\norm{\U^\perp}{H^2(\Omega)}\leq\cnorm\enorm{\U^\perp}{\mesh}$ and
\begin{equation}
\norm{\u-\U}{H^2(\Omega)}\leq\enorm{\u-\U^0}{}+\norm{\U^\perp}{H^2(\Omega)}
\leq\enorm{\u-\U^0}{}+Ch_\mesh^s\norm{\Lap\u}{H^s(\Omega)}
\end{equation}
\end{proof}
\section{A posteriori estimates}\label{sec:post}	
In this section we will derive the a posteriori error estimates for \eqref{eq:dnp} needed to to ensure the convergence of solutions $\U$ generated by the iterative procedure \eqref{eq:afem} to the the weak solution $\u$ of \eqref{eq:cwp}.
\subsection{The global upper bound}
We prove that the proposed error estimator is reliable.
\begin{lemma}[Estimator reliability]\label{lem:er}
Let $\mesh$ be a partition of $\Omega$ satisfying Conditions \eqref{eq:sr}.
The module $\ms{ESTIMATE}$ produces a posteriori error estimate $\eta_\mesh$ for the discrete error such that for a constants $C_\rm{rel,1},C_\rm{rel,2}>0$,
\begin{equation}\label{eq:res:lem:er}
\begin{split}
\ap(\u-\U,\u-\U)&\leq C_\rm{rel,1}\eta_\mesh^2(\U,\Omega)+\norm{\Ep}{H^{-2}(\Omega)}\\
&\quad+C_\rm{rel,2}\left(\gamma_1\norm{\U}{3/2,\mesh}^2+\gamma_2\left\norm{\textstyle\dfrac{\U}{\nu}\right}{1/2,\mesh}^2\right),
\end{split}
\end{equation}
with constants depending only on $\cshape$.
\end{lemma}
\begin{proof}
In this proof we will derive a localized quantification for the residual $\scr{R}_\mesh$.
In view of definition \eqref{eq:residual} we follow standard procedure and integrate by parts to obtain
\begin{equation}
\begin{split}
\inner{\scr{R}_\mesh,\v}&=\sum_{\tau\in\mesh}\left(
\int_\tau (f-\Op\U)\v
+\int_{\diff\tau}\Lap\U{\textstyle\dfrac{\v}{\n_\tau}}
-\int_{\diff\tau}{\textstyle\dfrac{\Lap\U}{\n_\tau}\v}\right)\\
&\quad-\int_\Gamma\left(\Lap\U{\textstyle\dfrac{\v}{\nu}}+{\textstyle\dfrac{\U}{\nu}}\Pi(\Lap\v)\right)
-\gamma_1\int_\Gamma h_\Gamma^{-3}\U\v\\
&\quad+\int_\Gamma\left(\textstyle{\dfrac{\Lap\U}{\nu}\v+\U\dfrac{\Pi(\Lap\v)}{\nu}}\right)
-\gamma_2\int_\Gamma h_\Gamma^{-1}{\textstyle\dfrac{\U}{\nu}\dfrac{\v}{\nu}}.
\end{split}
\end{equation}
We express all the integrals over cell boundaries as integrals over edges,
\begin{equation}
\begin{split}
\sum_{\tau\in\mesh}\left(\int_{\diff\tau}\Lap\U{\textstyle\dfrac{\v}{\n_\tau}}
-\int_{\diff\tau}{\textstyle\dfrac{\Lap\U}{\n_\tau}\v}\right)
&-\int_\Gamma\Lap\U{\textstyle\dfrac{\v}{\nu}}
+\int_\Gamma{\textstyle{\dfrac{\Lap\U}{\nu}\v}},\\
&=\sum_{\sigma\in\edges_\mesh}\left(\int_\sigma\jump{\textstyle\dfrac{\Lap\U}{\n_\sigma}}{\sigma}\v
-\int_\sigma\jump{\Lap\U}{\sigma}{\textstyle\dfrac{\v}{\n_\sigma}}\right).
\end{split}
\end{equation}
Asymptotic consistency of the discretization; see \eqref{eq:res:lem:ac}, enables us to express
\begin{equation}
\ap(\u-\U,\v)=\ap(\u-\U,\v-\Ip\v)+\inner{\Ep,\Ip\v},
\end{equation}
which will provide a sharp upper-bound estimate for residual.
We have
\begin{equation}\label{eq1:lem:er}
\begin{split}
|\!\inner{\scr{R}_\mesh,\v}\!|\leq&\sum_{\tau\in\mesh}\norm{f-\Op\U}{L^2(\tau)}\norm{\v-\Ip\v}{L^2(\tau)}
+\sum_{\sigma\in\edges_\mesh}\left\norm{\jump{\textstyle\dfrac{\Lap\U}{\n_\sigma}}{\sigma}\right}{L^2(\sigma)}\left\norm{\textstyle(\v-\Ip\v)\right}{L^2(\sigma)}\\
&+\sum_{\sigma\in\edges_\mesh}\norm{\jump{\Lap\U}{\sigma}}{L^2(\sigma)}\left\norm{{\textstyle\dfrac{}{\n_\sigma}}(\v-\Ip\v)\right}{L^2(\sigma)}
+\left|\int_\Gamma{\textstyle\dfrac{\U}{\nu}}\Pi[\Lap(\v-\Ip\v)]\right|
+|\inner{\Ep,\Ip\v}|\\
&+\left|\int_\Gamma\U{\textstyle\dfrac{\Pi[\Lap(\v-\Ip\v)]}{\nu}}\right|
+\gamma_1\left|\int_\Gamma h_\Gamma^{-3}\U(\v-\Ip\v)\right|
+\gamma_2\left|\int_\Gamma h_\Gamma^{-1}{\textstyle\dfrac{\U}{\nu}\dfrac{}{\nu}(\v-\Ip\v)}\right|.
\end{split}
\end{equation}
We define the interior residual $R=f-\Op\U$ and jump terms $J_1=\textstyle\jump{\dfrac{\Lap\U}{\n_\sigma}}{\sigma}$ and $J_2=\jump{\Lap\U}{\sigma}$ across  each interior edge $\sigma$.
Starting with the first three terms in \eqref{eq1:lem:er}, we use the approximation results from Lemma \ref{lem:ie} to estimate interior residual terms
\begin{equation}
\begin{split}
\sum_{\tau\in\mesh}\norm{R}{L^2(\tau)}\norm{\v-\Ip\v}{L^2(\tau)}
&\leq\sum_{\tau\in\mesh}\norm{R}{L^2(\tau)} \cproj h_\tau^2\norm{\v}{L^2(\omega_\tau)},\\
&\leq \cproj \bigg(\sum_{\tau\in\mesh}h_\tau^4\norm{R}{L^2(\tau)}^4\bigg)^{1/2}
\bigg(\sum_{\tau\in\mesh}\norm{\v}{H^2(\omega_\tau)}^2\bigg)^{1/2}.
\end{split}
\end{equation}
As for the interior edge jump terms,
\begin{equation}
\begin{split}
\sum_{\sigma\in\edges_\mesh}\left\norm{J_1\right}{L^2(\sigma)}\left\norm{\textstyle(\v-\Ip\v)\right}{L^2(\sigma)}
&\leq\sum_{\sigma\in\edges_\mesh}\left\norm{J_1\right}{L^2(\sigma)}\cproj h_\sigma^{3/2}\norm{\v}{L^2(\omega_\sigma)},\\
&\leq \cproj \bigg(\sum_{\sigma\in\edges_\mesh}h_\sigma^3\norm{J_1}{L^2(\sigma)}^2\bigg)^{1/2}
\bigg(\sum_{\sigma\in\edges_\mesh}\norm{\v}{H^2(\omega_\sigma)}^2\bigg)^{1/2},
\end{split}
\end{equation}
and
\begin{equation}
\begin{split}
\sum_{\sigma\in\edges_\mesh}\left\norm{J_2\right}{L^2(\sigma)}\left\norm{{\textstyle\dfrac{}{\n_\sigma}}(\v-\Ip\v)\right}{L^2(\sigma)}
&\leq\sum_{\sigma\in\edges_\mesh}\left\norm{J_2\right}{L^2(\sigma)}\cproj h_\sigma^{1/2}\norm{\v}{L^2(\omega_\sigma)},\\
&\leq \cproj \bigg(\sum_{\sigma\in\edges_\mesh}h_\sigma\norm{J_2}{L^2(\sigma)}^2\bigg)^{1/2}
\bigg(\sum_{\sigma\in\edges}\norm{\v}{H^2(\omega_\sigma)}^2\bigg)^{1/2}.
\end{split}
\end{equation}
From the finite-intersection property \eqref{eq:sr} we have $\sum_{\sigma\in\edges}
\norm{\v}{H^2(\omega_\sigma)}^2\leq \cshape\norm{\v}{H^2(\Omega)}^2$.
Proceeding to the treatment of domain boundary integrals,
\begin{equation}
\begin{split}
\left|\int_\Gamma{\textstyle\dfrac{\U}{\nu}}\Pi[\Lap(\v-\Ip\v)]\right|
&\leq\sum_{\sigma\in\bedges_\mesh}\left\norm{\textstyle\dfrac{\U}{\n_\sigma}\right}{L^2(\sigma)}\norm{\Pi[\Lap(\v-\Ip\v)]}{L^2(\sigma)},\\
&\leq\cdtrace\sum_{\sigma\in\bedges_\mesh}\left\norm{\textstyle\dfrac{\U}{\n_\sigma}\right}{\sigma}h_\sigma^{-1/2}\norm{\Lap(\v-\Ip\v)}{L^2(\tau(\sigma))},\\
&\leq\cproj \cdtrace\bigg(\sum_{\sigma\bedges_\mesh}h_\sigma^{-1}\left\norm{\textstyle\dfrac{\U}{\n_\sigma}\right}{\sigma}^2\bigg)^{1/2}
\bigg(\sum_{\sigma\in\bedges_\mesh}\norm{\v}{H^2(\omega_\tau)}^2\bigg)^{1/2},
\end{split}
\end{equation}
where $\tau(\sigma)$ is the boundary adjacent cell with edge $\sigma$.
Similarly,
\begin{equation}
\begin{split}
\left|\int_\Gamma\U{\textstyle\dfrac{\Pi[\Lap(\v-\Ip\v)]}{\nu}}\right|
&\leq\sum_{\sigma\in\bedges_\mesh}\norm{\U}{L^2(\sigma)}\left\norm{\textstyle\dfrac{}{\n_\sigma}\Pi[\Lap(\v-\Ip\v)]\right}{L^2(\sigma)},\\
&\leq\cinv\cdtrace\sum_{\sigma\in\bedges_\mesh}\norm{\U}{L^2(\sigma)}h_\sigma^{-3/2}\norm{\Lap(\v-\Ip\v)}{L^2(\tau(\sigma))},\\
&\leq \cproj \cinv\cdtrace\bigg(\sum_{\sigma\in\bedges_\mesh}h_\sigma^{-3}\norm{\U}{L^2(\sigma)}^2\bigg)^{1/2}\bigg(\sum_{\sigma\in\bedges_\mesh}\norm{\v}{H^2(\tau(\sigma))}^2\bigg)^{1/2}.
\end{split}
\end{equation}
and
\begin{equation}
\begin{split}
\left|\int_\Gamma h_\Gamma^{-3}\U(\v-\Ip\v)\right|
&\leq\sum_{\sigma\in\bedges_\mesh}h_\sigma^{-3/2}\norm{\U}{L^2(\sigma)}h_\sigma^{-3/2}\norm{\v-\Ip\v}{L^2(\sigma)},\\
&\leq\bigg(\sum_{\sigma\in\bedges_\mesh}h_\sigma^{-3}\norm{\U}{L^2(\sigma)}^2\bigg)^{1/2}\bigg(\sum_{\sigma\in\bedges_\mesh}h_\sigma^{-3}\norm{\v-\Ip\v}{L^2(\sigma)}^2\bigg)^{1/2},\\
&\leq \cproj \norm{\U}{3/2,\mesh}\bigg(\sum_{\sigma\in\bedges_\mesh}\norm{\v}{H^2(\omega_\sigma)}^2\bigg)^{1/2},
\end{split}
\end{equation}
and finally,
\begin{equation}
\left|\int_\Gamma h_\Gamma^{-1}{\textstyle\dfrac{\U}{\nu}\dfrac{}{\nu}(\v-\Ip\v)}\right|
\leq \cproj \left\norm{\textstyle\dfrac{\U}{\nu}\right}{1/2,\mesh}\bigg(\sum_{\sigma\in\bedges_\mesh}\norm{\v}{H^2(\omega_\sigma)}^2\bigg)^{1/2}.
\end{equation}
Invoking finite-intersection property \eqref{eq:sr} we have $\sum_{\sigma\in\bedges}\norm{\v}{H^2(\tau(\sigma))}^2\leq \cshape\norm{\v}{H^2(\Omega)}^2$.
Summing up we arrive at
\begin{equation}
\begin{split}
|\!\inner{\scr{R}_\mesh,\v}\!|&\leq \cproj \left\{\bigg(\sum_{\tau\in\mesh}h_\tau^4\norm{R}{L^2(\tau)}^2\bigg)^{1/2}
+\bigg(\sum_{\sigma\in\edges_\mesh}h_\sigma^3\left\norm{J_1\right}{L^2(\sigma)}^2\bigg)^{1/2}\right.\\
&\quad\left.+\bigg(\sum_{\sigma\in\edges_\mesh}h_\sigma\norm{J_2}{L^2(\sigma)}^2\bigg)^{1/2}
\right\}\norm{\v}{H^2(\Omega)}+\norm{\Ep}{H^{-2}(\Omega)}\norm{\v}{H^2(\Omega)}\\
&\quad+\cproj \left\{(\cinv\cdtrace+\gamma_1)\norm{\U}{3/2,\mesh}+(\cdtrace+\gamma_2)\left\norm{\textstyle\dfrac{\U}{\nu}\right}{1/2,\mesh}\right\}\norm{\v}{H^2(\Omega)}
\end{split}
\end{equation}
So we have
\begin{equation}
\begin{split}
\frac{|\!\inner{\scr{R}_\mesh,\v}\!|}{\norm{\v}{H^2(\Omega)}}&\leq\cproj\eta_\mesh(\U,\Omega)+\norm{\Ep}{H^{-2}(\Omega)}\\
&+\cproj\left\{(\cinv\cdtrace+\gamma_1)\norm{\U}{3/2,\mesh}+(\cdtrace+\gamma_2)\left\norm{\textstyle\dfrac{\U}{\nu}\right}{1/2,\mesh}\right\}
\end{split}
\end{equation}
\end{proof}

\subsection{Global lower bound}
We will define extension operators $E_\sigma:C(\sigma)\to C(\tau)$ for all edges $\sigma$ with $\sigma\subset\diff\tau$.
Let $\hat{\tau}=[0,1]\times[0,1]$ and $\hat{\sigma}=[0,1]\times\{0\}$.
Let $F_\tau:\bb{R}^2\to\bb{R}^2$ be the affine transformation comprising of translation and scaling mapping $\hat{\tau}$ onto $\tau$ and $\hat{\sigma}$ onto $\sigma$.
Define $\hat{E}:C(\hat{\sigma})\to C(\hat{\tau})$ via
\begin{equation}
\hat{E}v(x,y)=v(x)\quad \forall x\in\hat{\sigma},\quad(x,y)\in\hat{\tau},\quad v\in C(\hat{\sigma}).
\end{equation}
To this end, let $\sigma$ be an edge of a cell $\tau\in\mesh$, then define $E_\sigma:C(\sigma)\to C(\tau)$ via
\begin{equation}
E_\sigma v:=[\hat{E}(v\circ F_\tau)]\circ F_\tau^{-1}.
\end{equation}
In other words extending the values of $\v$ from $\sigma$ into $\tau$ along inward $\n_\sigma$.
Let $\psi_\tau$ be any smooth cut-off function with the following properties:
\begin{equation}\label{eq:interiorbubble}
\supp\psi_\tau\subseteq\tau,
\quad\psi_\tau\ge0,
\quad\max_{x\in\tau}\psi_\tau(x)\leq1.
\end{equation}
Furthermore, we will define for $\sigma=\diff\tau_1\cap\diff\tau_2=:D_\sigma$ two $C^1$ cut-off functions $\psi_\sigma$ and $\chi_\sigma$ via
with the following
\begin{equation}
\quad\supp\chi_\sigma\subseteq D_\sigma,
\quad\supp\psi_\sigma\subseteq D_\sigma,
\quad\psi_\tau\ge0,
\quad\max_{x\in\tau}\psi_\tau(x)\leq1,
\end{equation}
such that
\begin{equation}
\textstyle\ \dfrac{\psi_\sigma}{\n_\sigma}\equiv0,
\quad\chi_\sigma\equiv0
\quad \text{and}\quad d_1h_\sigma^{-1}\psi_\sigma\leq\dfrac{\chi_\sigma}{\n_\sigma}\leq d_2h_\sigma^{-1}\psi_\sigma\quad \text{along}\ \sigma.
\end{equation}
Let $\hat{\sigma}$ and $\hat{\tau}$ retain the same meanings as before.
Let $\hat{\tau}_1=\tau$ and let $\hat{\tau}_2=\{(x,y)\in\bb{R}^2:(x,-y)\in\hat{\tau}\}$ and let $\n=(0,-1)$
Now let $F_\sigma:\bb{R}^2\to\bb{R}^2$ be the affine map that maps $\hat{D}$ onto $D_\sigma$ and define
\begin{equation}\label{eq:edgebubbles}
\psi_\sigma=\hat{\psi}\circ F_\sigma^{-1},\quad
\chi_\sigma=\hat{\chi}\circ F_\sigma^{-1}.
\end{equation}
unit normal vector on $\hat{\sigma}$.
Let $g(y)$ be the cubic polynomial satisfying
\begin{equation}
g^{(4)}(y)=0\ \text{for}\ y\in(0,1)
\ \text{such that}\
g'(0)=g(1)=g'(1)=0\ \text{and}\ g(0)=1.
\end{equation}
Put $\phi(y)=g(y)\ds{1}_{[0,1]}+g(-y)\ds{1}_{[-1,0]}$.
Now define $G$ to be the quartic polynomial satisfying
\begin{equation}
G^{(5)}(y)=0\ \text{for}\ y\in(0,1)\ \text{such that}\
G(0)=G(1)=G'(1)=0\
\text{and}\ G'(0)=\phi(0).
\end{equation}
Put $\eta(y)=G(y)\ds{1}_{[0,1]}-G(-y)\ds{1}_{[-1,0]}$.
Finally, let $H=x^2(1-x)^2(1-y^2)^2$ and let $\Phi(x,y)=H(x,y)\ds{1}_{y\ge0}+H(x,-y)\ds{1}_{y\leq 0}$.
Finally set
\begin{equation}
\hat{\psi}(x,y)=\Phi(x,y)\phi(y),\quad\hat{\chi}(x,y)=\Phi(x,y)\eta(y),\quad(x,y)\in\hat{D}
\end{equation}
\begin{lemma}[Localizing estimates]\label{lem:LocalizingEstimates}
Let $\mesh$ be a partition of $\Omega$ satisfying Conditions \eqref{eq:sr}.
Let $\tau$ be a cell in partition $\mesh$.
For a constant $c_m>0$ depending only on polynomial degree $m$,
\begin{equation}\label{eq1:result:lem:LocalizingEstimates}
\norm{q}{L^2(\tau)}^2\leq \cc\int_\tau\psi_\tau q^2\quad\forall q\in\bb{P}_m(\tau),
\end{equation}
Let $\sigma$ be an edge in $\edges_\mesh$ and let $\tau_1$ and $\tau_2$ be cells from $\mesh$ for which $\sigma\subseteq\overline{\tau_1}\cap\overline{\tau_2}$.
We also have
\begin{equation}\label{eq2:result:lem:LocalizingEstimates}
\norm{q}{L^2(\sigma)}^2\leq \cc\int_\sigma\psi_\sigma q^2
\end{equation}
and
\begin{equation}\label{eq3:result:lem:LocalizingEstimates}
\norm{\psi_\sigma^{1/2} E_\sigma q}{L^2(\tau)}\leq \cc h_\sigma^{1/2}\norm{q}{L^2(\sigma)}\quad\forall\tau\in D_\sigma
\end{equation}
holding for every $q\in\bb{P}_m(\sigma)$.
\end{lemma}
%
\begin{proof}
Relations \eqref{eq1:result:lem:LocalizingEstimates} and \eqref{eq2:result:lem:LocalizingEstimates} are proven in the same fashion as in \cite{}.
We focus on \eqref{eq3:result:lem:LocalizingEstimates}.
We prove that $q\mapsto\norm{\psi_\sigma^{1/2} E_\sigma q}{L^2(\tau)}$ is a norm on $\bb{P}_m(\sigma)$.
\begin{equation}
\begin{split}
\norm{E_\sigma q}{L^2(\tau)}
&=|\tau|^{1/2}\norm{\hat{E}(q\circ F_\tau)}{L^2(\hat{\tau})}.
\end{split}
\end{equation}
It clear that $\hat{q}\in\bb{P}_m(\hat{\sigma})$ is identically zero if and only if its extension $\hat{E}\hat{q}$ is identically zero on $\hat{\tau}$.
So $\hat{q}\mapsto\norm{\hat{E}\hat{q}}{L^2(\hat{\tau})}$ is an equivalent norm on $\bb{P}_m(\hat{\sigma})$ we have
\begin{equation}
\norm{\hat{E}(q\circ F_\tau)}{L^2(\hat{\tau})}\leq c\norm{q\circ F_\tau}{L^2{\hat{\sigma}}}=ch_\sigma^{-1/2}\norm{q}{L^2(\sigma)}
\end{equation}
so with $c|\tau|^{1/2}h_\sigma^{1/2}\leq\overline{c}h_\sigma^{1/2}$
\begin{equation}
\norm{\psi_\sigma^{1/2}E_\sigma q}{L^2(\tau)}\leq\overline{c}h_\sigma^{1/2}\norm{q}{L^2(\sigma)}.
\end{equation}
\end{proof}
%
\begin{lemma}[Estimator efficiency]\label{lem:EstimatorEfficiency}
Let $\mesh$ be a partition of $\Omega$ satisfying conditions \eqref{eq:sr}.
The module $\ms{ESTIMATE}$ produces a posteriori error estimate of the discrete solution error such that
\begin{equation}\label{eq:result:lem:EstimatorEfficiency}
\Ceff\,\eta_\mesh^2(\U,\Omega)\leq\enorm{\u-\U}{\mesh}^2+\osc_\mesh^2(\Omega).
\end{equation}
with constant $\Ceff$ depending only on $\cshape$.
\end{lemma}
\begin{proof}
The proof is carried by localizing the error contributions coming from the cells residuals $R$ and edge jumps $J_1$ and $J_2$.
For $\tau\in\mesh$ let $\psi_\tau\in H^2_0(\tau)$ be as in \eqref{eq:bubble} and let $\overline{R}$ be a polynomial approximation of $R|_\tau$ by means of an $L^2$-orthogonal projection.
Using the norm-equivalence relation \eqref{eq1:result:lem:LocalizingEstimates} of Lemma
\ref{lem:LocalizingEstimates}
\begin{equation}\label{eq1:lem:EstimatorEfficiency}
\norm{\overline{R}}{L^2(\tau)}^2\leq \cc\int_\tau\overline{R}(\overline{R}\psi_\tau)\leq \cc\norm{\overline{R}}{H^{-2}(\tau)}\norm{\overline{R}\psi_\tau}{H^2(\tau)}.
\end{equation}
From \eqref{eq:inve:lem:ie} and \eqref{eq:interiorbubble}, $\norm{\overline{R}\psi_\tau}{H^2(\tau)}\leq \cc h_\tau^{-2}\norm{\overline{R}}{L^2(\tau)}$ with a constant $\cc>0$ which depends on the polynomial degree $r$ and \eqref{eq1:lem:EstimatorEfficiency} now reads $h_\tau^{2}\norm{\overline{R}}{L^2(\tau)}\leq \cc\norm{\overline{R}}{H^{-2}(\tau)}$.
Then
\begin{equation}\label{eq:interior:lem:EstimatorEfficiency}
\begin{split}
h_\tau^2\norm{R}{L^2(\tau)}&\leq h_\tau^2\norm{\overline{R}}{L^2(\tau)}+h_\tau^2\norm{R-\overline{R}}{L^2(\tau)},\\
&\leq \cc\norm{\overline{R}}{H^{-2}(\tau)}+h_\tau^2\norm{R-\overline{R}}{L^2(\tau)},\\
&\leq \cc\norm{R}{H^{-2}(\tau)}+\cc\norm{R-\overline{R}}{H^{-2}(\tau)}+h_\tau^2\norm{R-\overline{R}}{L^2(\tau)},\\
&\lapprox \cc\norm{\scr{R}}{H^{-2}(\tau)}+h_\tau^2\norm{R-\overline{R}}{L^2(\tau)}.
\end{split}
\end{equation}
Recognizing that $R-\overline{R}=f-\overline{f}$, define $\osc_\mesh(\tau):=h^2_\tau\norm{f-\overline{f}}{L^2(\tau)}$.
We turn our attention to the jump terms across the interior edges.
We begin with the edge residual $J_1$.
Let an edge $\sigma\in\edges_\mesh$ and cells $\tau_1,\tau_2\in\mesh$ be such that
$\sigma\subset\diff\tau_1\cap\diff\tau_2$ and denote $D_\sigma=\overline{\tau_1\cup\tau_2}$.
If $\v\in H^2_0(D_\sigma)$ then
\begin{equation}\label{eq2:lem:EstimatorEfficiency}
\inner{\scr{R},\v}=\int_{D_\sigma}R\v+\int_\sigma J_1\v-\int_\sigma J_2\textstyle\dfrac{\v}{\n_\sigma}.
\end{equation}
Let $\psi_\sigma$ be the bubble function \eqref{eq:edgebubbles} and constantly extend the values of $J_1$ in directions $\pm\n_\sigma$; i.e, into each of $\tau_i$, and set $v_\sigma=\psi_\sigma J_1$.
Then \eqref{eq2:lem:EstimatorEfficiency} reads
\begin{equation}\label{eq3:lem:EstimatorEfficiency}
\cc\norm{J_1}{L^2(\sigma)}^2\leq\int_\sigma\psi_\sigma J_1^2=\inner{\scr{R},\v_\sigma}-\int_{D_\sigma}R\v_\sigma.
\end{equation}
From \eqref{eq:dtrace:lem:ie} and \eqref{eq:edgebubbles} we have the estimates $\norm{\psi_\sigma E_\sigma J_1}{H^2(D_\sigma)}\leq \cinv h_\sigma^{-2}\norm{E_\sigma J_1}{L^2(D_\sigma)}$ and $\norm{\psi_\sigma E_\sigma J_1}{L^2(D_\sigma)}\leq \cc h_\sigma^{1/2}\norm{J_1}{L^2(\sigma)}$
which we apply to \eqref{eq3:lem:EstimatorEfficiency} to obtain
\begin{equation}\label{eq:jump1:lem:EstimatorEfficiency}
\begin{split}
\cc\norm{J_1}{L^2(\sigma)}^2&\leq\left(\cinv h_\sigma^{-2}\norm{\scr{R}}{H^{-2}(D_\sigma)}+\norm{R}{L^2(D_\sigma)}\right)\norm{J_1}{L^2(D_\sigma)},\\
&\leq\cdtrace h_\sigma^{1/2}\left(\cinv h_\sigma^{-2}\norm{\scr{R}}{H^{-2}(D_\sigma)}+\norm{R}{L^2(D_\sigma)}\right)\norm{J_1}{L^2(\sigma)},
\end{split}
\end{equation}
where the the last line follows from \eqref{eq3:result:lem:LocalizingEstimates}.
Now let $\chi_\sigma$ be the function \eqref{eq2:result:lem:LocalizingEstimates}, extend the values of $J_2$ into $D_\sigma$ and set $\w_\sigma=\chi_\sigma J_2$.
We then have
\begin{equation}
\inner{\scr{R},\w_\sigma}=\int_{D_\sigma}R\w_\sigma-h_\sigma^{-1}\int_\sigma\psi_\sigma J_2^2.
\end{equation}
Similarly, we obtain
\begin{equation}\label{eq:jump2:lem:EstimatorEfficiency}
\cc h_\sigma^{-1}\norm{J_2}{L^2(\sigma)}^2\leq\cdtrace h_\sigma^{1/2}\left(\cinv h_\sigma^{-2}\norm{\scr{R}}{H^{-2}(D_\sigma)}+\norm{R}{L^2(D_\sigma)}\right)\norm{J_2}{L^2(\sigma)}.
\end{equation}
We have from \eqref{eq:jump1:lem:EstimatorEfficiency} and
\eqref{eq:jump2:lem:EstimatorEfficiency}
\begin{equation}
\begin{split}
h_\sigma^3\norm{J_1}{L^2(\sigma)}^2+h_\sigma\norm{J_2}{L^2(\sigma)}^2
&\lapprox{\textstyle\frac{\cinv\cdtrace}{\cc}}\norm{\scr{R}}{H^{-2}(D_\sigma)}^2
+{\textstyle\frac{\cdtrace}{\cc}} h_\sigma^4\norm{R}{L^2(D_\sigma)}^2.
\end{split}
\end{equation}
Summing up we have
\begin{equation}
\begin{split}
\eta_\mesh^2(\V,\tau)&=h_\tau^4\norm{R}{L^2(\tau)}^2
+\sum_{\sigma\in\diff\tau}\left(h_\sigma^3\norm{J_1}{L^2(\sigma)}^2+h_\sigma\norm{J_2}{L^2(\sigma)}^2\right),\\
&\leq h_\tau^4\norm{R}{L^2(\tau)}^2+{\textstyle\frac{\cdtrace}{\cc}}\sum_{\sigma\in\edges_\mesh}
\left(\cinv\norm{\scr{R}}{H^{-2}(D_\sigma)}^2+h_\sigma^4\norm{R}{L^2(D_\sigma)}^2\right),\\
&\leq (1+\cshape)h_\tau^4\norm{R}{L^2(\omega_\tau)}^2+{\textstyle\frac{\cinv\cdtrace}{\cc}}\sum_{\sigma\in\diff\tau}\norm{\scr{R}}{H^{-2}(D_\sigma)}^2,\\
\end{split}
\end{equation}
we arrive at
\begin{equation}
\begin{split}
\eta_\mesh^2(\V,\tau)&\lapprox(1+\cshape)\left(\cc\norm{\scr{R}}{H^{-2}(\tau)}^2+\osc^2_\mesh(\omega_\tau)\right)
+{\textstyle\frac{\cinv\cdtrace}{\cc}}\sum_{\sigma\in\diff\tau}\norm{\scr{R}}{H^{-2}(D_\sigma)}^2.
\end{split}
\end{equation}
Note that
\begin{equation}
\begin{split}
\sum_{\sigma\in\edges_\mesh}\norm{\scr{R}}{H^{-2}(D_\sigma)}^2
&\leq\Ccont\sum_{\sigma\in\edges_\mesh}\norm{\u-\U}{H^2(D_\sigma)}^2
\leq\cshape\Ccont\norm{\u-\U}{H^2(\Omega)}^2
\end{split}
\end{equation}
\end{proof}

\subsection{Discrete upper bound}

\begin{lemma}[Estimator discrete reliability]\label{lem:dre}
Let $\mesh$ be a partition of $\Omega$ satisfying conditions \eqref{eq:sr} and let $\mesh_\ast=\ms{REFINE}\,[\mesh,R]$ for some refined set $R\subseteq\mesh$.
If $\U$ and $\U_\ast$ are the respective solutions to \eqref{eq:cdp} on $\mesh$ and $\mesh_\ast$, then for a constants $ C_\rm{dRel,1}, C_\rm{dRel,2}>0$, depending only on $\cshape$,
\begin{equation}\label{eq:result:lem:dre}
\begin{split}
\enorm{\U_\ast-\U}{\mesh_\ast}^2&\leq C_\rm{dRel,1}\eta^2_\mesh(\U,\omega_{R_{\mesh\to\mesh_\ast}})+\norm{\Epp}{H^{-2}(\Omega)}\\
&\quad+C_\rm{dRel,2}\left(\gamma_1\norm{\U}{3/2,R}+\gamma_2\left\norm{\textstyle\dfrac{\U}{\nu}\right}{1/2,R}\right),
\end{split}
\end{equation}
where $\omega_{R_{\mesh\to\mesh_\ast}}$ is understood as the union of support extensions of refined cells from $\mesh$ to obtain $\mesh_\ast$.
\end{lemma}

\begin{proof}
Let $\e_\ast=\U_\ast-\U$. First note that if $\V\in\Xp$ then in view of the nesting $\Xp\subset\Xpp$, $\app(\U_\ast,\V)=\ellf(\V)=\ap(\U,\V)$ and
\begin{equation}\label{eq1:lem:dre}
\app(\U_\ast-\U,\V)=[\ap(\U,\V)-\app(\U,\V)]-\inner{\Epp,\V}.
\end{equation}
Then
\begin{equation}
\app(\U_\ast-\U,\e_\ast)=\app(\U_\ast-\U,\e_\ast-\V)+[\ap(\U,\V)-\app(\U,\V)]-\inner{\Epp,\V}.
\end{equation}
We treat each of the three terms separately.
For the first term, we form disconnected subdomains $\Omega_i\subseteq\Omega$, $i\in J$, each formed from the interiors of connected components of $\Omega_\ast=\cup_{\tau\in R_{\mesh\to\mesh_\ast}}\overline{\tau}$.
Then to each subdomain $\Omega_i$ we form a partition $\mesh_i=\{\face\in\mesh:\face\subset\Omega_i\}$, interior edges $\edges_i=\{\sigma\in\edges_\mesh:\sigma\subset\diff\tau,\ \tau\in\mesh_i\}$ and boundary edges $\bedges_i=\{\sigma\in\bedges_\mesh:\sigma\subset\diff\tau,\ \tau\in\mesh_i\}$, and a corresponding finite-element space $\X_i$.
Let $\V\in\Xp$ be an approximation of $\e_\ast$ be given by
\begin{equation}
\V=\e_\ast\ds{1}_{\Omega\backslash\Omega_\ast}+\sum_{i\in J}(\Ip\e_\ast)\cdot\ds{1}_{\Omega_i}.
\end{equation}
Then $e_\ast-\V\equiv0$ on $\Omega\backslash\Omega_\ast$ and performing  integration by parts will yield
\begin{equation}
\begin{split}
\app(\U_\ast-\U,\e_\ast-\V)=&\sum_{i\in J}
\bigg[\sum_{\tau\in\mesh_i}\inner{R,\e_\ast-\Ip\e_\ast}_\tau
+\sum_{\sigma\in\edges_i}\left\{\inner{J_1,\e_\ast-\Ip\e_\ast}_\sigma+\inner{J_2,\e_\ast-\Ip\e_\ast}_\sigma\right\}\\
&+\sum_{\sigma\in\bedges_i}
\left(\int_\sigma\U{\textstyle\dfrac{}{\n_\sigma}}\left[\Pip\Lap(\e_\ast-\Ip\e_\ast)\right]
-\int_\sigma{\textstyle\dfrac{\U}{\n_\sigma}}\Pip\Lap(\e_\ast-\Ip\e_\ast)\right.\\
&\left.-\,\gamma_1 h_\sigma^{-3}\int_\sigma\U(\e_\ast-\Ip\e_\ast)-\gamma_2 h_\sigma^{-1}\int_\sigma{\textstyle\dfrac{\U}{\n_\sigma}}{\textstyle\dfrac{}{\n_\sigma}}(\e_\ast-\Ip\e_\ast)\right)\bigg],
\end{split}
\end{equation}
Following the same procedure carried in Lemma \ref{lem:er} we have
\begin{equation}
\begin{split}
\sum_{\tau\in\mesh_i}\inner{R,\e_\ast-\Ip\e_\ast}_\tau
&+\sum_{\sigma\in\edges_i}\left\{\inner{J_1,\e_\ast-\Ip\e_\ast}_\sigma
+\inner{J_2,\e_\ast-\Ip\e_\ast}_\sigma\right\}\\
&\leq\cproj\bigg(\sum_{\tau\in\mesh_i}\eta^2_\mesh(\U,\tau)\bigg)^{1/2}
\bigg(\sum_{\tau\in\mesh_i}\norm{\e_\ast}{H^2(\omega_\tau)}^2\bigg)^{1/2},
\end{split}
\end{equation}
and
\begin{equation}
\begin{split}
\sum_{\sigma\in\bedges_i}&
\left(\int_\sigma\U{\textstyle\dfrac{}{\n_\sigma}}\left[\Pip\Lap(\e_\ast-\Ip\e_\ast)\right]
-\int_\sigma{\textstyle\dfrac{\U}{\n_\sigma}}\Pip\Lap(\e_\ast-\Ip\e_\ast)\right)\\
&\leq\cproj\left\{\cinv\cdtrace\norm{\U}{3/2,\mesh}+\cdtrace\left\norm{\textstyle\dfrac{\U}{\nu}\right}{1/2,\mesh}\right\}\bigg(\sum_{\sigma\in\bedges_i}\norm{\e_\ast}{H^2(\omega_\sigma)}^2\bigg)^{1/2},
\end{split}
\end{equation}
from which we obtain
\begin{equation}\label{eq2:lem:dre}
\app(\U_\ast-\U,\e_\ast-\V)\lapprox\sum_{i\in J}\left(\eta_{\mesh}(\U,\Omega_i)
+\gamma_1\norm{\U}{3/2,\mesh_i}+\gamma_2\left\norm{\textstyle\dfrac{\U}{\nu}\right}{1/2,\mesh_i}\right)\norm{\e_\ast}{H^2(\omega_i)}.
\end{equation}
where $\omega_i=\cup\{\omega_\tau:\tau\in\Omega_i\}$, and the implicit constant depends on $\cproj$, $\cshape$ and discrete inverse constants $\cinv$ and $\cdtrace$.
The second term reads
\begin{equation}
\begin{split}
\ap(\U,\V)-\app(\U,\V)&=\sum_{\sigma\in\bedges_\mesh}\left(\gamma_1h_\sigma^{-3}\int_\sigma\U\V+\gamma_2h_\sigma^{-1}\int_\sigma\textstyle\dfrac{\U}{\n_\sigma}\dfrac{\V}{\n_\sigma}\right)\\
&\quad-\sum_{\sigma_\ast\in\bedges_{\mesh_\ast}}\left(\gamma_1h_{\sigma_\ast}^{-3}\int_{\sigma_\ast}\U\V+\gamma_2h_{\sigma_\ast}^{-1}\int_{\sigma_\ast}\textstyle\dfrac{\U}{\n_{\sigma_\ast}}\dfrac{\V}{\n_{\sigma_\ast}}\right).
\end{split}
\end{equation}
Using the grading $h_\sigma\leq\cshape h_{\sigma_\ast}$ whenever $\sigma_{\ast,i}\subset\sigma$, $\sigma_{\ast,i}\in\bedges_{\mesh_\ast}$ and $\sigma\in\bedges_\mesh$
\begin{equation}
h^{-3}_\sigma\int_\sigma\U\V-h_{\sigma_{\ast,1}}^{-3}\int_{\sigma_{\ast,1}}\U\V-h_{\sigma_{\ast,2}}^{-3}\int_{\sigma_{\ast,2}}\U\V
=(h^{-3}_\sigma-h_{\sigma_{\ast,1}}^{-3}-h_{\sigma_{\ast,2}}^{-3})\int_\sigma\U\V
\lapprox \cshape h^{-3}_\sigma\int_\sigma\U\V
\end{equation}
so we arrive at
\begin{equation}
|\ap(\U,\V)-\app(\U,\V)|\lapprox \cshape\left(\gamma_1\norm{\U}{3/2,R}+\gamma_2\left\norm{\textstyle\dfrac{\U}{\nu}\right}{1/2,R}\right)\norm{\e_\ast}{H^2(\Omega)}.
\end{equation}
%
%
Set $\omega_{R_{\mesh\to\mesh_\ast}}=\cup\{\omega_\tau:\tau\in R_{\mesh\to\mesh_\ast}\}$.
We therefore have
\begin{equation}
\begin{split}
\enorm{\U_\ast-\U}{\mesh_\ast}^2&\leq\Cdrel\left(\eta_\mesh(\U,\omega_{R_{\mesh\to\mesh_\ast}})+\norm{\Epp}{H^{-2}(\Omega)}\right)\norm{\e_\ast}{H^2(\Omega)}\\
&\quad+\left(\gamma_1\norm{\U}{3/2,R}+\gamma_2\left\norm{\textstyle\dfrac{\U}{\nu}\right}{1/2,R}\right)\norm{\e_\ast}{H^2(\Omega)}
\end{split}
\end{equation}
\end{proof}

\begin{remark}
In the case $\Xp\subset\Ho(\Omega)$, the relation in \eqref{eq1:lem:dre} would reduce to zero on the right-hand side owing to Galerkin orthogonality, saving us from carrying the treatment subsequent to\eqref{eq2:lem:dre}.
\end{remark}

With the following realization from \cite{al2018adaptivity} we have shown that contributions from domain boundary integrals are dominated by the those coming from the mesh interior.
\begin{lemma}\label{lem:boundarycontrol}
For sufficiently large stabilization terms $\gamma_1$ and $\gamma_2$,
\begin{equation}
(\gamma_1-C_\rm{R})\norm{\U}{3/2,\mesh}^2+(\gamma_2-C_\rm{R})\left\norm{\textstyle\dfrac{\U}{\nu}\right}{1/2,\mesh}^2
\leq\Ccoer^{-1}\eta_\mesh^2(\U,\Omega)
\end{equation}
with $C_\rm{R}\lapprox\frac{\cshape}{\Ccoer}$.
\end{lemma}
\begin{corollary}\label{cor:aposteriori}
Under the assumptions of lemma \ref{lem:er} and lemma \ref{lem:dre},
\begin{equation}
\ap(\u-\U,\u-\U)\leq\Crel\etap^2(\U,\Omega)+\norm{\Ep}{H^{-2}(\Omega)}^2,
\end{equation}
and
\begin{equation}
\enorm{\U-\U_\ast}{\mesh_\ast}^2\leq\Cdrel\eta_\mesh^2(\U,\omega_{R_{\mesh\to\mesh_\ast}})+\norm{\Ep}{H^{-2}(\Omega)}^2.
\end{equation}
\end{corollary}
\section{Convergence} 		
In section we show that the derived computable estimator \eqref{eq:estimator} when used to direct refinement will result in decreased error.
This will hinge on the estimator Lipschitz property of Lemma \ref{lem:elp}.
To show that procedure \eqref{eq:afem} exhibits convergence we must be able to relate the errors of consecutive discrete solutions.
In the conforming discrete method \eqref{eq:cdp} the symmetry of the bilinear form, consistency of the formulation and  finite-element spline space nesting will readily provide that via Galerkin Pythagoras in Lemma \ref{lem:gp}.
This is not the case in Nitsche's formulation \eqref{eq:ndp} since our formulation is no longer consistent with \eqref{eq:cwp}.
The proposed workaround is inspired by \cite{bonito2010quasi}.
%
\subsection{Error reduction} 
\begin{lemma}[Estimator Lipschitz property]\label{lem:elp}
Let $\mesh$ be a partition of $\Omega$ satisfying conditions \eqref{eq:sr}.
There exists a constant $\Clip>0$, depending only $\cshape$, such that for any cell $\face\in\mesh$ we have
\begin{equation}\label{eq:result:lem:elp}
|\eta_\mesh (\V,\face)-\eta_\mesh (\W,\face)|
\leq\Clip\semi{\V-\W}{H^2(\omega_\face)},
\end{equation}
holding for every pair of finite-element splines $\V$ and $\W$ in $\X_\mesh$.
\end{lemma}
\begin{proof}
Let $\V$ and $\W$ be finite-element splines in $\X_\mesh$ and let $\face$ be a cell in partition $\mesh$.
\begin{equation}\label{eq1:lem:elp}
\begin{split}
\eta_\mesh(\V,\face)-\eta_\mesh(\W,\face)&=
h_\face^2\left(\norm{\eff-\Op\V}{L^2(\face)}
-\norm{\eff-\Op\W}{L^2(\face)}\right)\\
&+\sum_{\sigma\subset\diff\tau}h_\sigma^{1/2}\left(
\norm{\jump{\Lap\V}{\sigma}}{L^2(\sigma)}
-\norm{\jump{\Lap\W}{\sigma}}{L^2(\sigma)}
\right)\\
&+\sum_{\sigma\subset\diff\tau}h_\sigma^{3/2}\left(
\norm{\jump{\diffe\Lap\V}{\sigma}}{L^2(\edge)}
-\norm{\jump{\diffe\Lap\W}{\sigma}}{L^2(\edge)}
\right).
\end{split}
\end{equation}
Treating the interior term,
\begin{equation}
\begin{split}
\norm{\eff-\Op\V}{L^2(\face)}-\norm{\eff-\Op\W}{L^2(\face)}
&\leq\semi{\V-\W}{H^4(\face)}
\leq\cinv h^{-2}_\face\semi{\V-\W}{H^2(\face)}.
\end{split}
\end{equation}
Treating the edge terms we have
\begin{equation}
\norm{\jump{\Lap\V}{\sigma}}{L^2(\sigma)}
-\norm{\jump{\Lap\W}{\sigma}}{L^2(\sigma)}
\leq\norm{\jump{\Lap\V-\Lap\W}{\sigma}}{L^2(\sigma)}.
\end{equation}
Let $\face'$ from $\mesh$ be a cell that shares the edge $\sigma$, i.e $\tau'$ is an adjacent cell to $\face$.
For any finite-element spline $\V\in\Xp$ we have
\begin{equation}
\norm{\jump{\V}{\edge}}{\edge}
\leq\cdtrace\left(
h_\sigma^{-1/2}\norm{\V}{\face}+h_\sigma^{-1/2}\norm{\V}{\face'}
\right)
\leq\cdtrace h_\sigma^{-1/2}\norm{\V}{\omega_\face}.
\end{equation}
Replacing $\V$ with $\Lap\V-\Lap\W$ gives
\begin{equation}
h_\sigma^{1/2}\norm{\jump{\Lap\V-\Lap\W}{\sigma}}{\sigma}
\leq\cinv\cdtrace\semi{\V-\W}{H^2(\omega_\face)}.
\end{equation}
Similarly, we have
\begin{equation}
h_\sigma^{3/2}\left(
\norm{\jump{\diffe\Lap\V}{\sigma}}{L^2(\edge)}
-\norm{\jump{\diffe\Lap\W}{\sigma}}{L^2(\edge)}
\right)
\leq\cinv\cdtrace\semi{\V-\W}{H^2(\omega_\face)}.
\end{equation}
It then follows from \eqref{eq1:lem:elp}
\begin{equation}
\begin{split}
|\eta_\mesh(\V,\tau)-\eta_\mesh(\W,\tau)|&\leq
\cinv(\semi{\V-\W}{H^2(\face)}+2\cdtrace\semi{\V-\W}{H^2(\omega_\face)}),\\
&\leq\cinv(1+2\cdtrace)\semi{\V-\W}{H^2(\omega_\face)}.
\end{split}
\end{equation}
\end{proof}

\begin{lemma}[Estimator error reduction]\label{lem:err}
Let $\mesh$ be a partition of $\Omega$ satisfying conditions \eqref{eq:sr},
let $\marked\subseteq\mesh$
and let $\mesh_\ast=\mathbf{REFINE}\,[\mesh,\marked]$.
There exists constants $\lambda\in(0,1)$ and $\Cest>0$, depending only on $\cshape$,
such that for any $\delta>0$ it holds that for any pair of finite-element splines $\V\in\X_\mesh$ and $\V_\ast\in\X_{\mesh_\ast}$
we have
\begin{equation}\label{eq:result:lem:err}
\eta_{\mesh_\ast}^2(\V_\ast,\Omega)
\leq(1+\delta)\left\{\eta_\mesh^2(\V,\Omega)-{\textstyle\frac{1}{2}}\eta_\mesh^2(\V,\marked)\right\}+\cshape(1+{\textstyle\frac{1}{\delta}})\enorm{\V-\V_\ast}{}^2.
\end{equation}
\end{lemma}
\begin{proof}
Let $\marked\subseteq\mesh$ be a set of marked elements from partition $\mesh$
and let $\mesh_\ast=\mathbf{REFINE}\,[\mesh,\marked]$.
For notational simplicity we denote $\bb{X}_{\mesh_\ast}$ and $\eta_{\mesh_\ast}$ by $\X_\ast$ and $\eta_\ast$, respectively.
Let $\V$ and $\V_\ast$ be the respective finite-element splines from $\Xp$ and $\X_\ast$.
Let $\face$ be a cell from partition $\mesh_\ast$.
In view of the Lipschitz property of the estimator (Lemma \ref{lem:elp})
and the nesting $\Xp\subseteq\X_\ast$,
\begin{equation}
\eta_{\ast}^2(\V_\ast,\face)
\lapprox\eta_{\ast}^2(\V,\face)+\semi{\V-\V_\ast}{H^2(\omega_\face)}^2
+2\eta_{\ast}(\V_\ast,\face)\semi{\V-\V_\ast}{H^2(\omega_\face)}.
\end{equation}
Given any $\delta>0$, an application of Young's inequality on the last term gives
\begin{equation}
2\eta_{\ast}(\V_\ast,\face)\semi{\V-\V_\ast}{H^2(\omega_\face)}
\leq\delta\eta_{\ast}^2(\V_\ast,\face)+{\textstyle\frac{1}{\delta}}\semi{\V-\V_\ast}{H^2(\omega_\face)}^2.
\end{equation}
We now have
\begin{equation}
\eta_{\ast}^2(\V_\ast,\face)
\lapprox(1+\delta)\eta_{\ast}^2(\V,\face)
+(1+{\textstyle\frac{1}{\delta}})\semi{\V-\V_\ast}{H^2(\omega_\face)}^2.
\end{equation}
Recalling that the partition cell are disjoint with uniformly bounded support extensions, we may sum over all the cells $\face\in\mesh_\ast$ to obtain
\begin{equation}
\eta_\ast^2(\V_\ast,\mesh_\ast)\leq(1+\delta)\eta_\ast^2(\V,\mesh_\ast)+\cshape(1+{\textstyle\frac{1}{\delta}})\enorm{\V-\V_\ast}{}^2.
\end{equation}
It remains to estimate $\eta_\ast^2(\V,\mesh_\ast)$. Let $|\marked|$ be the sum areas of all cells in $\marked$.
For every marked element $\face\in\marked$ define $\mesh_{\ast,\marked}=\{\rm{child}(\face):\face\in\marked\}$.
Let $b>0$ denote the number of bisections required to obtain the conforming partition $\mesh_\ast$ from $\mesh$.
Let $\face_\ast$ be a child of a cell $\face\in\marked$.
Then $h_{\face_\ast}\leq2^{-1}h_\face$.
Noting that $\V\in\Xp$ we have no jumps within $\face$
\begin{equation}
\eta_\ast^2(\V,\tau_\ast)=h_{\tau_\ast}^4\norm{f-\Op\V}{\tau_\ast}^2\leq(2^{-1}h_\tau)^4\norm{f-\Op\V}{\tau_\ast}^2,
\end{equation}
summing over all children
\begin{equation}
\sum_{\tau_\ast\in\rm{children}(\tau)}\eta_\ast^2(\V,\tau_\ast)
\leq2^{-1}\eta_\mesh^2(\V,\tau),
\end{equation}
and we obtain by disjointness of partitions an estimate on the error reduction
\begin{equation}
\sum_{\face_\ast\in\mesh_{\ast,M}}\eta_{\ast}^2(\V,\face_\ast)\leq2^{-1}\est_\mesh^2(\V,\marked).
\end{equation}
For the remaining cells $T\in\mesh\backslash\marked$, the estimator monotonicity implies $\eta_{\mesh_\ast}(\V,T)\leq\eta_\mesh(\V,T)$.
Decompose the partition $\mesh$ as union of marked cells in $\marked$ and their complement $\mesh\backslash\marked$ to conclude the total error reduction obtained by $\ms{REFINE}$ and the choice of Dorfler parameter $\theta$
\begin{equation}
\eta_{\mesh_\ast}^2(\V,\Omega)\leq\estP^2(\V,\Omega\backslash\marked)+2^{-1}\estP^2(\V,\marked)
=\estP^2(\V,\Omega)-{\textstyle\frac{1}{2}}\eta^2_\mesh(\V,\marked).
\end{equation}
\end{proof}

\subsection{Conforming discretization}
\begin{lemma}[Galerkin Pythaguras]\label{lem:gp}
Let $\mesh$ and $\mesh_\ast$ be partitions of $\Omega$ satisfying conditions \eqref{eq:sr} with $\mesh_\ast\ge\mesh$
and let $\U\in\Xp$ and $\U_\ast\in\X_{\mesh_\ast}$ be the spline solutions to \eqref{eq:cdp}.
Then
\begin{equation}\label{eq:result:lem:gp}
\enorm{\u-\U_\ast}{}^2=\enorm{\u-\U_\ast}{}^2-\enorm{\U_\ast-\U}{}^2.
\end{equation}
\end{lemma}
\begin{proof}
At first we express
\begin{equation}\label{eq1:lem:gp}
\begin{split}
\a(\u-\U_\ast,\u-\U_\ast)&=\a(\u-\U_\ast,\u-\U+\U-\U_\ast),\\
&=\a(\u-\U,\u-\U)-\a(\upp-\up,\upp-\up)\\
&+\a(\up-\upp,\u-\upp)+\a(\u-\U_\ast,\U-\U_\ast).
\end{split}
\end{equation}
Recognizing that
\begin{equation}
\a(\u-\U_\ast,\U-\U_\ast)=\a(\U-\U_\ast,\u-\U)=0,
\end{equation}
we arrive at
\begin{equation}
\a(\u-\U_\ast,\u-\U_\ast)=\a(\u-\U,\u-\U)-\a(\upp-\up,\upp-\up).
\end{equation}
\end{proof}

\begin{theorem}[Convergence of conforming AFEM]\label{thm:ConvConf}
For a contractive factor $\alpha\in(0,1)$ and a constant $\Cest>0$, given any successive mesh partitions $\mesh$ and $\mesh_\ast$ satisfying conditions \eqref{eq:sr}, $f\in L^2(\Omega)$ and Dolfer parameter $\theta\in(0,1]$, the adaptive procedure $\mathbf{AFEM}\,[\mesh,f,\theta]$ with produce two successive solutions $\U\in\Xp$ and $\U_\ast\in\Xpp$ to problem \eqref{eq:cdp} for which
\begin{equation}\label{eq:result:thm:ConvConf}
\enorm{\u-\U_\ast}{}^2+\Cest\eta_{\mesh_\ast}^2(\U_\ast,\Omega)\leq\alpha\big(\enorm{\u-\U}{}^2+\Cest\eta_{\mesh}^2(\U,\Omega)\big).
\end{equation}
\end{theorem}
\begin{proof}
Adopt the following abbreviations:
\begin{eqnarray}
e_\mesh=\enorm{\u-\U}{},&E_\ast=\enorm{\U_\ast-\U}{},\\
\eta_\mesh=\eta_{\mesh}(\U,\Omega),&\eta_\mesh(\marked)=\eta_{\mesh}(\U,\marked).
\end{eqnarray}
Define constants $\qest(\theta,\delta):=(1+\delta)(1-{\frac{\theta^2}{2}})<1$ and $\Cest^{-1}:=\cshape(1+{\frac{1}{\delta}})>0$ so that in view of Dorlfer $-\eta_\mesh^2(\marked)\leq-\theta^2\eta_\mesh^2$,
\begin{equation}
(1+\delta)\left\{\eta_\mesh^2(\Omega)-{\textstyle\frac{1}{2}}\eta_\mesh^2(\marked)\right\}\leq q_\rm{est}\eta_\mesh^2,
\end{equation}
and \eqref{eq:result:lem:EstRed} reads $\eta_{\mesh_\ast}^2\leq\qest\eta_\mesh^2+\Cest^{-1} E_\ast^2$.
Galerkin orthogonality and estimator error reduction estimate \eqref{eq:result:lem:EstRed}
\begin{equation}
\begin{split}
e_{\mesh_\ast}^2+\Cest\eta_{\mesh_\ast}^2&\leq e_\mesh^2-E^2_\ast+\Cest\left(\qest\eta_\mesh^2+\Cest^{-1} E_\ast^2\right)=e_{\mesh}^2+\qest\Cest\eta_\mesh^2
\end{split}
\end{equation}
Let $\alpha$ be a positive parameter and express $e_{\mesh}^2=\alpha e_{\mesh}^2+(1-\alpha)e_{\mesh}^2$. Invoking on the reliability estimate $e_{\mesh}^2\leq \Crel\eta_\mesh^2$ on one of the decomposed terms gives
\begin{equation}
e_{\mesh_\ast}^2+\qest\Cest\eta_{\mesh_\ast}^2\leq\alpha e_{\mesh}^2+\left[(1-\alpha)\Crel+\qest\Cest\right]\eta_\mesh^2.
\end{equation}
Choose $\delta>0$ with $\delta<\frac{\theta^2}{2-\theta^2}$ so that $q_\rm{est}(\theta,\delta)\in(0,1)$ and we may choose a contractive $\alpha<1$ for which $(1-\alpha)\Crel+\qest\Cest\leq\textstyle\alpha \Cest$.
Indeed,
\begin{equation}
(1-\alpha)\Crel+\qest\Cest\leq\textstyle\alpha \Cest\iff\frac{\qest\Cest+\Crel}{\Cest+\Crel}\leq\alpha
\end{equation}
and $\alpha_\rm{max}(\theta,\delta):=\frac{\qest\Cest+\Crel}{\Cest+\Crel}<1$ so pick $\alpha\in(\alpha_\rm{max}(\theta,\delta),1)$.
\end{proof}
\subsection{Nitsche's discretization}
\begin{lemma}[Mesh perturbation]\label{lem:MP}
Let $\mesh$ and $\mesh_\ast$ be successive partitions satisfying conditions \eqref{eq:sr} which are obtained by $\ms{REFINE}$.
Then for a constant $\Ccomp>0$, depending only on $\cshape$,
we have for any $\delta>0$
\begin{equation}
\app(\v,\v)\leq(1+4\delta\Ccoer)\ap(\v,\v)+\frac{\Ccomp}{\delta}\left(\gamma_1\norm{\v}{3/2,\mesh}^2+\gamma_2\left\norm{\textstyle\dfrac{\v}{\nu}\right}{1/2,\mesh}^2\right),
\end{equation}
holding for every function $\v\in H^2(\Omega)$.
\end{lemma}
\begin{proof}
Given any $\v\in H^2(\Omega)$, we write
\begin{equation}\label{eq1:lem:MP}
\begin{split}
\app(\v,\v)=&\ap(\v,\v)+2\bigg(\int_\Gamma\Pip(\Lap\v){\textstyle\dfrac{\v}{\nu}}
-\int_\Gamma{\textstyle\dfrac{\Pip(\Lap\v)}{\nu}}\v\bigg)
-\gamma_1\left(\norm{\v}{\mesh,3/2}^2-\norm{\v}{\mesh_\ast,3/2}^2\right)\\
&-2\bigg(\int_\Gamma\Pipp(\Lap\v){\textstyle\dfrac{\v}{\nu}}
-\int_\Gamma{\textstyle\dfrac{\Pipp(\Lap\v)}{\nu}}\v\bigg)
-\gamma_2\left(\left\norm{\textstyle\dfrac{\v}{\nu}\right}{\mesh,1/2}^2
-\left\norm{\textstyle\dfrac{\v}{\nu}\right}{\mesh_\ast,1/2}^2\right).
\end{split}
\end{equation}
Look at the boundary integral terms depending on $\mesh$. Let $\sigma\in\bedges_\mesh$, which would be an edge to some cell $\tau\in\mesh$.
\begin{equation}\label{eq2:lem:MP}
\begin{split}
\int_\sigma\Pip(\Lap\v)\dfrac{\v}{\n_\sigma}&\leq\norm{\Pip(\Lap\v)}{\sigma}\left\norm{\textstyle\dfrac{\v}{\n_\sigma}\right}{\sigma}
\leq\cdtrace\cproj h_\sigma^{-1/2}\norm{\Lap\v}{\tau}\left\norm{\textstyle\dfrac{\v}{\n_\sigma}\right}{\sigma}.
\end{split}
\end{equation}
Summing \eqref{eq2:lem:MP} over all $\sigma\in\bedges_\mesh$ and an application of Schwarz's inequality on the summation would give
\begin{equation}
\begin{split}
\bigg|\int_\Gamma\Pip(\Lap\v){\textstyle\dfrac{\v}{\nu}}\bigg|&
\lapprox\bigg(\sum_{\sigma\in\bedges_\mesh}h_\sigma^{-1}\left\norm{\textstyle\dfrac{\v}{\n_\sigma}\right}{\sigma}^2\bigg)^{1/2}
\bigg(\sum_{\tau\in\mesh:\diff\tau\cap\Gamma\neq\emptyset}\norm{\Lap\v}{\tau}^2\bigg)^{1/2}\\
&\leq\left\norm{\textstyle\dfrac{\v}{\nu}\right}{\mesh,1/2}\norm{\Lap\v}{L^2(\Omega)}
\end{split}
\end{equation}
Similarly, using the inverse-estimate $\norm{\dfrac{\Pip(\Lap\v)}{\n_\sigma}}{\sigma}\leq\cinv h_\sigma^{-1}\norm{\Pip(\Lap\v)}{\sigma}$, we obtain
\begin{equation}
\bigg|\int_\Gamma{\textstyle\dfrac{\Pip(\Lap\v)}{\nu}}\v\bigg|\leq\cdtrace\cinv\cproj\norm{\v}{\mesh,3/2}\norm{\Lap\v}{L^2(\Omega)}.
\end{equation}
We carry the same reasoning for the remaining counterpart boundary integral.
Employing Young's inequality with $\delta>0$ we arrive at
\begin{equation}\label{eq3:lem:MP}
\begin{split}
\app(\v,\v)&\lapprox\ap(\v,\v)
+4\delta\norm{\Lap\v}{L^2(\Omega)}^2
+\left(\textstyle\frac{1}{\delta}+\gamma_1\right)\norm{\v}{\mesh,3/2}^2
+\left(\textstyle\frac{1}{\delta}+\gamma_1\right)\norm{\v}{\mesh_\ast,3/2}^2\\
&+\left(\textstyle\frac{1}{\delta}+\gamma_2\right)\left\norm{\textstyle\dfrac{\v}{\nu}\right}{\mesh,1/2}^2
+\left(\textstyle\frac{1}{\delta}+\gamma_2\right)\left\norm{\textstyle\dfrac{\v}{\nu}\right}{\mesh_\ast,1/2}^2.
\end{split}
\end{equation}
With the fact that $h_\sigma\leq\cshape h_{\sigma_\ast}$, with $\sigma\in\bedges_\mesh$ and $\sigma_\ast\in\bedges_{\mesh_\ast}$, we infer that $\norm{\v}{3/2,\mesh_\ast}\leq\cshape^{-1}\norm{\v}{\mesh,3/2}$ and $\norm{\dfrac{\v}{\nu}}{1/2,\mesh_\ast}\leq\cshape^{-1}\norm{\dfrac{\v}{\nu}}{1/2,\mesh}$.
\begin{equation}
\left(\textstyle\frac{1}{\delta}+\gamma_1\right)\left(\norm{\v}{\mesh_\ast,3/2}^2
+\norm{\v}{\mesh_\ast,3/2}^2\right)\leq\frac{\Ccomp\gamma_1}{\delta}\norm{\v}{\mesh,3/2}^2,
\end{equation}
where $\Ccomp>0$ is an appropriate proportionality parameter that depends on $\cshape$.
A similar argument holds for terms including boundary norms of $\dfrac{\v}{\nu}$.
\end{proof}
\begin{lemma}[Comparison of solutions]\label{lem:cos}
Let $\mesh$ and $\mesh_\ast$ be successive admissible partitions obtained by $\ms{REFINE}$
and let $\U\in\Xp$ and $\U_\ast\in\X_{\mesh_\ast}$ be the finite-element spline solutions to \eqref{eq:ndp}.
Then we have for any $\eps>0$
\begin{equation}\label{eq:result:lem:cos}
\begin{split}
\app(\epp,\epp)\leq&(1+\eps)\ap(\ep,\ep)
-\frac{\Ccoer}{2}\enorm{\upp-\up}{\mesh_\ast}^2\\
&+\frac{8C_\rm{Comp}}{\eps\min\{\gamma_1,\gamma_2\}}\eta_\mesh^2
+\frac{2\norm{\Epp}{H^{-2}(\Omega)}^2}{\Ccoer}
\end{split}
\end{equation}
\end{lemma}
\begin{proof}
We begin from an analogous expression to \eqref{eq1:lem:gp}.
\begin{equation}
\begin{split}
\app(\u-\upp,\u-\upp)&=\app(\u-\upp,\u-\up+\up-\upp)\\
&=\app(\u-\up,\u-\up)+\app(\up-\upp,\u-\up)+\app(\u-\upp,\up-\upp)
\end{split}
\end{equation}
In view of Lemma \ref{lem:MP}, the first term
\begin{equation}
\begin{split}
\app(\u-\up,\u-\up)\leq&(1+4\delta\Ccoer)\ap(\u-\up,\u-\up)\\
&+\frac{\Ccomp}{\delta}\left(\gamma_1\norm{\up}{3/2,\mesh}^2+\gamma_2\left\norm{\textstyle\dfrac{\up}{\nu}\right}{1/2,\mesh}^2\right).
\end{split}
\end{equation}
In view of inconsistency, the third term
\begin{equation}
\app(\u-\upp,\up-\upp)=\inner{\Epp,\up-\upp}
\end{equation}
The second term reads, thanks to the symmetry of the bilinear form,
\begin{equation}
\begin{split}
\app(\up-\upp,\u-\up)&=\app(\up-\upp,\upp-\up)+\app(\up-\upp,\u-\upp)\\
&=-\app(\upp-\up,\upp-\up)+\inner{\Epp,\up-\upp}
\end{split}
\end{equation}
We arrive at
\begin{equation}
\begin{split}
\app(\u-\upp,\u-\upp)\leq&(1+4\delta\Ccoer)\ap(\u-\up,\u-\up)-\app(\upp-\up,\upp-\up)\\
&+2\inner{\Epp,\up-\upp}
+\frac{\Ccomp}{\delta}\left(\gamma_1\norm{\up}{3/2,\mesh}^2+\gamma_2\left\norm{\textstyle\dfrac{\up}{\nu}\right}{1/2,\mesh}^2\right).
\end{split}
\end{equation}
An application of Young's inequality to
\begin{equation}
2\inner{\Epp,\up-\upp}
\leq\frac{2\norm{\Epp}{H^{-2}(\Omega)}^2}{\Ccoer}+\frac{\Ccoer}{2}\enorm{\up-\upp}{\mesh}^2
\end{equation}
and in view of Lemma \ref{lem:boundarycontrol} we obtain
\begin{equation}
\begin{split}
\app(\epp,\epp)&\leq(1+4\delta\Ccoer)\ap(\ep,\ep)
-\frac{\Ccoer}{2}\enorm{\upp-\up}{\mesh_\ast}^2\\
&+\frac{C_\rm{Comp}}{\delta\min\{\gamma_1,\gamma_2\}}\eta_\mesh^2
+\frac{2\norm{\Epp}{H^{-2}(\Omega)}^2}{\Ccoer}
\end{split}
\end{equation}
Given $\delta>0$ pick $\eps=4\delta\Ccoer$.
\end{proof}

\begin{theorem}[Convergence of Nitsche's AFEM]\label{thm:nc}
For a contractive factor $\alpha\in(0,1)$ and a constant $\Cest>0$, given any successive mesh partitions $\mesh$ and $\mesh_\ast$ satisfying conditions \eqref{eq:sr}, $f\in L^2(\Omega)$ and Dolfer parameter $\theta\in(0,1]$, the adaptive procedure $\mathbf{AFEM}\,[\mesh,f,\theta]$ with produce two successive solutions $\U\in\Xp$ and $\U_\ast\in\Xpp$ to problem \eqref{eq:ndp} for which
\begin{equation}\label{eq:result:thm:ConvNit}
\end{equation}
\end{theorem}
\begin{proof}
Adopt the following abbreviations:
\begin{eqnarray}
\a_\mesh=\a_\mesh(\u-\U,\u-\U),&E_\ast=\enorm{\U-\U_\ast}{\mesh_\ast},\\
\eta_\mesh=\eta_{\mesh}(\U,\mesh),&\eta_\mesh(\marked)=\eta_{\mesh}(\U,\marked).
\end{eqnarray}
Let $\gamma=\min\{\gamma_1,\gamma_2\}$.
In view of Lemma \ref{lem:cos},
\begin{equation}\label{eq1:thm:nc}
\app+\Cest\etapp^2\leq(1+\eps)\ap-{\textstyle\frac{\Ccoer}{2}}E_\ast^2
+\textstyle{\frac{\Ccomp}{\eps\gamma}}\etap^2+{\textstyle\frac{2}{\Ccoer}}\norm{\Epp}{H^{-2}(\Omega)}^2+\Cest\etapp^2.
\end{equation}
By invoking Lemma \ref{lem:err} on $\Cest\etapp^2$
\begin{equation}
\begin{split}
\app+\Cest\etapp^2&\leq(1+\eps)\ap
+\textstyle{\frac{\Ccomp}{\eps\gamma}}\etap^2+{\textstyle\frac{2}{\Ccoer}}\norm{\Epp}{H^{-2}(\Omega)}^2\\
&\quad+\Cest\left[(1+\delta)\left\{\etap^2-{\textstyle\frac{1}{2}}\etap^2(\marked)\right\}+\Cest^{-1}E_\ast^2\right]
\end{split}
\end{equation}
and eliminate $E_\ast$ from the previous expression.
From Dorler $-\etap^2(M)\leq\theta^2\etap^2$
and in view of Corollary \ref{cor:aposteriori},
\begin{equation}
\begin{split}
\Cest(1+\delta)\left\{\etap^2-{\textstyle\frac{1}{2}}\etap^2(\marked)\right\}
&\leq\Cest(1+\delta)\etap^2-\Cest(1+\delta){\textstyle\frac{\theta^2}{2}}\etap^2\\
&\leq\Cest(1+\delta)\etap^2-\Cest(1+\delta){\textstyle\frac{\theta^2}{2}}\left(\textstyle\frac{1}{2}\etap^2+\frac{1}{2\Crel}\ap\right)
\end{split}
\end{equation}
\eqref{eq1:thm:nc} now reads
\begin{equation}
\app+\Cest\etapp^2\leq\left(1+\eps-\Cest(1+\delta){\textstyle\frac{\theta^2}{4\Crel}}\right)\ap
+\left(\textstyle{\frac{\Ccomp}{\eps\gamma}}+\Cest(1+\delta)\left(1-\frac{\theta^2}{4}\right)\right)\etap^2+{\textstyle\frac{2}{\Ccoer}}\norm{\Epp}{H^{-2}(\Omega)}^2.
\end{equation}
Observing that $\Cest(1+\delta)=\delta\cshape^{-1}$ we arrive at
\begin{equation}
\app+\Cest\etapp^2\leq\left(1+\eps-\delta{\textstyle\frac{\theta^2}{4\cshape\Crel}}\right)\ap
+\Cest\left(\textstyle{\frac{\Ccomp}{\eps\gamma\Cest}}+(1+\delta)\left(1-\frac{\theta^2}{4}\right)\right)\etap^2+{\textstyle\frac{2}{\Ccoer}}\norm{\Epp}{H^{-2}(\Omega)}^2.
\end{equation}
It what remains we verify the existence of $\eps>0,\delta>0$ and $\gamma>0$ such that the quantities $1+\eps-\delta{\textstyle\frac{\theta^2}{4\cshape\Crel}}$ and $(1+\delta)\left(1-\frac{\theta^2}{4}\right)$ are strictly positive and less that $1$.
Let $\Lambda_1=(\cshape\Crel)^{-1}$ and $\Lambda_2=\Ccomp\cshape$.
Then the corresponding conditions will read
\begin{equation}\label{eq1:lem:nc}
\textstyle
0<1+\eps-\delta\frac{\theta^2}{4}\Lambda_1<1
\quad\text{and}\quad
\frac{\Lambda_2}{\delta\eps\gamma}<\frac{\theta^2}{4}-\frac{1}{1+\frac{1}{\delta}}.
\end{equation}
Choosing $\eps=\delta\frac{\theta^2}{8}\Lambda_1$ and $\delta<\min\{1/(\frac{4}{\theta^2}-1),\frac{8}{\theta^2\Lambda_1}\}$ ensures that $1/{(1+\frac{1}{\delta}})<\frac{\theta^2}{4}$ and
\begin{equation}
\textstyle
0<1+\eps-\delta\frac{\theta^2}{4}\Lambda_1=1-\eps<1,
\end{equation}
then pick $\gamma>0$ sufficiently large to obtain the second relation in \eqref{eq1:lem:nc}.
\end{proof}

\section{Acknowledgements}
We thank Emmanuil Georgoulis for discussion about dG methods and his invaluable advice.
\bibliography{publications}

\begin{thebibliography}{10}

\bibitem{adams1975sobolev}
{\sc R.~A. Adams}, {\em Sobolev spaces. 1975}, Academic Press, New York, 1975.

\bibitem{ainsworth2011posteriori}
{\sc M.~Ainsworth and J.~T. Oden}, {\em A posteriori error estimation in finite
  element analysis}, vol.~37, John Wiley \& Sons, 2011.

\bibitem{al2018adaptivity}
{\sc I.~Al~Balushi, W.~Jiang, G.~Tsogtgerel, and T.-Y. Kim}, {\em Adaptivity of
  a b-spline based finite-element method for modeling wind-driven ocean
  circulation}, Computer Methods in Applied Mechanics and Engineering, 332
  (2018), pp.~1--24.

\bibitem{babuvska1978posteriori}
{\sc I.~Babu{\v{s}}ka and W.~C. Rheinboldt}, {\em A-posteriori error estimates
  for the finite element method}, International Journal for Numerical Methods
  in Engineering, 12 (1978), pp.~1597--1615.

\bibitem{babuvvska1978error}
{\sc I.~Babuv{\v{s}}ka and W.~C. Rheinboldt}, {\em Error estimates for adaptive
  finite element computations}, SIAM Journal on Numerical Analysis, 15 (1978),
  pp.~736--754.

\bibitem{bazilevs2006isogeometric}
{\sc Y.~Bazilevs, L.~Beirao~da Veiga, J.~A. Cottrell, T.~J. Hughes, and
  G.~Sangalli}, {\em Isogeometric analysis: approximation, stability and error
  estimates for h-refined meshes}, Mathematical Models and Methods in Applied
  Sciences, 16 (2006), pp.~1031--1090.

\bibitem{bazilevs2007weak}
{\sc Y.~Bazilevs and T.~J. Hughes}, {\em Weak imposition of dirichlet boundary
  conditions in fluid mechanics}, Computers \& Fluids, 36 (2007), pp.~12--26.

\bibitem{binev2004adaptive}
{\sc P.~Binev, W.~Dahmen, and R.~DeVore}, {\em Adaptive finite element methods
  with convergence rates}, Numerische Mathematik, 97 (2004), pp.~219--268.

\bibitem{bonito2010quasi}
{\sc A.~Bonito and R.~H. Nochetto}, {\em Quasi-optimal convergence rate of an
  adaptive discontinuous galerkin method}, SIAM Journal on Numerical Analysis,
  48 (2010), pp.~734--771.

\bibitem{bubuvska1984feedback}
{\sc I.~Bubu{\v{s}}ka and M.~Vogelius}, {\em Feedback and adaptive finite
  element solution of one-dimensional boundary value problems}, Numerische
  Mathematik, 44 (1984), pp.~75--102.

\bibitem{cascon2008quasi}
{\sc J.~M. Cascon, C.~Kreuzer, R.~H. Nochetto, and K.~G. Siebert}, {\em
  Quasi-optimal convergence rate for an adaptive finite element method}, SIAM
  Journal on Numerical Analysis, 46 (2008), pp.~2524--2550.

\bibitem{dorfler1996convergent}
{\sc W.~D{\"o}rfler}, {\em A convergent adaptive algorithm for poisson's
  equation}, SIAM Journal on Numerical Analysis, 33 (1996), pp.~1106--1124.

\bibitem{embar2010imposing}
{\sc A.~Embar, J.~Dolbow, and I.~Harari}, {\em Imposing dirichlet boundary
  conditions with nitsche's method and spline-based finite elements},
  International journal for numerical methods in engineering, 83 (2010),
  pp.~877--898.

\bibitem{feischl2014adaptive}
{\sc M.~Feischl, T.~Führer, and D.~Praetorius}, {\em Adaptive fem with optimal
  convergence rates for a certain class of nonsymmetric and possibly nonlinear
  problems}, SIAM Journal on Numerical Analysis, 52 (2014), pp.~601--625.

\bibitem{grisvard2011elliptic}
{\sc P.~Grisvard}, {\em Elliptic problems in nonsmooth domains}, vol.~69, SIAM,
  2011.

\bibitem{hughes2005isogeometric}
{\sc T.~J. Hughes, J.~A. Cottrell, and Y.~Bazilevs}, {\em Isogeometric
  analysis: Cad, finite elements, nurbs, exact geometry and mesh refinement},
  Computer methods in applied mechanics and engineering, 194 (2005),
  pp.~4135--4195.

\bibitem{juntunen2009nitsche}
{\sc M.~Juntunen and R.~Stenberg}, {\em Nitsche's method for general boundary
  conditions}, Mathematics of computation, 78 (2009), pp.~1353--1374.

\bibitem{kim2015b}
{\sc T.-Y. Kim, T.~Iliescu, and E.~Fried}, {\em B-spline based finite-element
  method for the stationary quasi-geostrophic equations of the ocean}, Computer
  Methods in Applied Mechanics and Engineering, 286 (2015), pp.~168--191.

\bibitem{morin2000data}
{\sc P.~Morin, R.~H. Nochetto, and K.~G. Siebert}, {\em Data oscillation and
  convergence of adaptive fem}, SIAM Journal on Numerical Analysis, 38 (2000),
  pp.~466--488.

\bibitem{morin2002convergence}
\leavevmode\vrule height 2pt depth -1.6pt width 23pt, {\em Convergence of
  adaptive finite element methods}, SIAM review, 44 (2002), pp.~631--658.

\bibitem{morin2008basic}
{\sc P.~Morin, K.~G. Siebert, and A.~Veeser}, {\em A basic convergence result
  for conforming adaptive finite elements}, Mathematical Models and Methods in
  Applied Sciences, 18 (2008), pp.~707--737.

\bibitem{nitsche1971variation}
{\sc J.~Nitsche}, {\em {\"U}ber ein variationsprinzip zur l{\"o}sung von
  dirichlet-problemen bei verwendung von teilr{\"a}umen, die keinen
  randbedingungen unterworfen sind}, 36 (1971), pp.~9--15.

\bibitem{scott1990finite}
{\sc L.~R. Scott and S.~Zhang}, {\em Finite element interpolation of nonsmooth
  functions satisfying boundary conditions}, Mathematics of Computation, 54
  (1990), pp.~483--493.

\bibitem{siebert2010converg}
{\sc K.~G. Siebert}, {\em A convergence proof for adaptive finite elements
  without lower bound}, IMA journal of numerical analysis, 31 (2010),
  pp.~947--970.

\bibitem{speleers2016effortless}
{\sc H.~Speleers and C.~Manni}, {\em Effortless quasi-interpolation in
  hierarchical spaces}, Numerische Mathematik, 132 (2016), pp.~155--184.

\bibitem{stenberg1995some}
{\sc R.~Stenberg}, {\em On some techniques for approximating boundary
  conditions in the finite element method}, Journal of Computational and
  applied Mathematics, 63 (1995), pp.~139--148.

\bibitem{stevenson2005optimal}
{\sc R.~Stevenson}, {\em An optimal adaptive finite element method}, SIAM
  journal on numerical analysis, 42 (2005), pp.~2188--2217.

\bibitem{verfurth1994posteriori}
{\sc R.~Verf{\"u}rth}, {\em A posteriori error estimation and adaptive
  mesh-refinement techniques}, Journal of Computational and Applied
  Mathematics, 50 (1994), pp.~67--83.

\bibitem{vuong2011hierarchical}
{\sc A.-V. Vuong, C.~Giannelli, B.~J{\"u}ttler, and B.~Simeon}, {\em A
  hierarchical approach to adaptive local refinement in isogeometric analysis},
  Computer Methods in Applied Mechanics and Engineering, 200 (2011),
  pp.~3554--3567.

\end{thebibliography}
\bibliographystyle{siam}

\end{document}